\documentclass[11pt]{amsart}
\usepackage{amsmath,amsfonts,amsthm,mathrsfs,amssymb,amscd,comment,enumerate,amsxtra,url,tikz-cd}
\input{xypic}
\xyoption{all}

\usepackage{xcolor} 
\colorlet{mdtRed}{red!50!black}
\definecolor{dblue}{rgb}{0,0,.6}
\usepackage[colorlinks]{hyperref}
\hypersetup{linkcolor=blue,citecolor=dblue,filecolor=dullmagenta,urlcolor=mdtRed}

\numberwithin{equation}{section}
\newtheorem{theorem}[equation]{Theorem}
\newtheorem{corollary}[equation]{Corollary}
\newtheorem{lemma}[equation]{Lemma}
\newtheorem{proposition}[equation]{Proposition}
\newtheorem{definition}[equation]{Definition}

\newtheorem*{theorem*}{Theorem}
\newtheorem*{corollary*}{Corollary}
\newtheorem*{proposition*}{Proposition}

\theoremstyle{remark}
\newtheorem{remark}[equation]{Remark}

\def\subsection{
	\refstepcounter{equation}
	\noindent {\bf \arabic{section}.\arabic{equation}.}
}

\newcommand{\Z}{\mathbb{Z}}
\newcommand{\C}{\mathbb{C}}

\newcommand{\ms}[1]{\mathscr{#1}}
\newcommand{\mb}[1]{\mathbb{#1}}
\newcommand{\mc}[1]{\mathcal{#1}}
\renewcommand{\t}[1]{\tilde{#1}}

\usepackage{marginnote}
\usetikzlibrary{calc}

\tikzset{
	symbol/.style={
		draw=none,
		every to/.append style={
			edge node={node [sloped, allow upside down, auto=false]{$#1$}}}
	}
}

\begin{document}
	
	\title[Picard Groups of Some Quot Schemes]
	{Picard Groups of Some Quot Schemes}

	\author[C. Gangopadhyay]{Chandranandan Gangopadhyay} 
	
	\address{Department of Mathematics, Indian Institute of Science Education and Research, Pune, 411008, Maharashtra, India.} 
	
	\email{chandranandan@iiserpune.ac.in} 
	
	\author[R. Sebastian]{Ronnie Sebastian} 
	
	\address{Department of Mathematics, Indian Institute of Technology Bombay, Powai, Mumbai 400076, Maharashtra, India.} 
	
	\email{ronnie@math.iitb.ac.in} 
	
	

	\begin{abstract}
	Let $C$ be a smooth projective curve over the 
	field of complex numbers $\C$ of genus $g(C)>0$. 
	Let $E$ be a 
	locally free sheaf on $C$ of rank $r$ and degree $e$.
	Let $\mc Q:={\rm Quot}_{C/\C}(E,k,d)$
	denote the Quot scheme of quotients of $E$ of rank $k$
	and degree $d$. For $k>0$ and $d\gg0$ 
	we compute the Picard group of $\mc Q$. 
	\end{abstract}

	\maketitle
\section{Introduction}

Let $C$ be a smooth projective curve over the 
field of complex numbers $\C$. We shall denote the 
genus of $C$ by $g(C)$. Throughout this article we shall
assume that $g(C)\geqslant 1$. 
Let $E$ be a 
locally free sheaf on $C$ of rank $r$ and degree $e$.
Throughout this article  
\begin{equation}\label{def Q}
	\mc Q:={\rm Quot}_{C/\C}(E,k,d)
\end{equation}
will denote the Quot scheme of quotients of $E$ of rank $k$
and degree $d$.

Stromme proved  
that $\mc Q_{\mb P^1/\C}(\mc O_{\mb P^1}^{\oplus n},k,d)$ 
is a smooth projective variety and computed its 
Picard group and nef cone. In \cite{Jow}, the author computes 
the effective cone of $\mc Q_{\mb P^1/\C}(\mc O_{\mb P^1}^{\oplus n},k,d)$. 
In \cite{Ito}, the author studies the birational geometry 
of $\mc Q_{\mb P^1/\C}(\mc O_{\mb P^1}^{\oplus n},k,d)$.
When $E$ is trivial and $g(C)\geqslant1$, the space $\mc Q$
is studied in \cite{BDW} and it is proved 
that when $d\gg0$ it is irreducible and generically smooth.
For $g(C)\geqslant 1$ and $E$ trivial, the divisor 
class group of $\mc Q$ was computed in \cite{HO} 
under the assumption $d\gg 0$. 


When $g(C)\geqslant 1$, it was proved in \cite{PR} 
that $\mc Q$ is irreducible and generically smooth when $d\gg 0$. See also 
\cite{Goller}, \cite{CCH-Isotropic}, \cite{CCH-Lagrangian} 
for similar results on other variations of this Quot scheme.
We use this as a starting point to further investigate the 
space $\mc Q$ when $d\gg0$ and compute 
its Picard group. In the case when $k=r-1$ we have that $\mc Q$ 
is a projective bundle over the Jacobian of $C$ for $d\gg 0$ 
(Theorem \ref{thm when k=r-1}) and as a result its Picard 
group can be computed easily (Corollary \ref{cor when k=r-1}).
In Theorem \ref{det flat}
we show that if $d\gg0$ then $\mc Q$ is an integral variety which
	is normal, a local complete intersection and locally factorial.  
We compute the 
Picard group of $\mc Q$ in the following cases. 
\begin{theorem}[Theorem \ref{picard group of Quot}]
	Let $k\leqslant r-2$. Assume one of the following two holds
	\begin{itemize}
		\item $k\geqslant 2$ and $g(C)\geqslant 3$, or
		\item $k\geqslant 3$ and $g(C)=2$.
	\end{itemize}
	Then for $d\gg 0$ we have 
	$${\rm Pic}(\mc Q)\cong {\rm Pic}({\rm Pic}^0(C))\times \Z \times \Z\,.$$
\end{theorem}
Note that we have a natural determinant map
$${\rm det}:\mc Q\longrightarrow {\rm Pic}^d(C)$$
which sends a quotient $[E\longrightarrow F]\mapsto {\rm det}(F)$. 
In Theorem \ref{det flat} we show that ${\rm det}$ is a flat map when $d\gg0$.
For $[L]\in {\rm Pic}^d(C)$ let $\mc Q_L$ be the scheme theoretic fiber 
of det over $[L]$. We prove the following analogous results for $\mc Q_L$.
\begin{theorem}[Theorem \ref{Q_L is locally factorial}]
	Let $k\geqslant 2,g(C)\geqslant 2$. Let $d\gg0$.
	Then $\mc Q_L$ is a local complete intersection, integral, 
	normal and locally factorial scheme.
\end{theorem}
\begin{theorem}[Theorem \ref{Picard group of Q_L}]
	Let $k\leqslant r-2$. Assume one of the following two holds
	\begin{itemize}
		\item $k\geqslant 2$ and $g(C)\geqslant 3$, or
		\item $k\geqslant 3$ and $g(C)=2$.
	\end{itemize}
	Let $d\gg0$. Then ${\rm Pic}(\mc Q_L)\cong \Z \times \Z$.
\end{theorem}
When $k=1$ the above results can be improved to the case $g(C)\geqslant1$.
In Theorem \ref{k=1} we show that if $d\gg0$ then 
${\rm Pic}(\mc Q)\cong {\rm Pic}({\rm Pic}^0(C))\times \Z \times \Z$
and ${\rm Pic}(\mc Q_L)\cong \Z \times \Z$.

We say a few words about how the above results are proved. 
By a very large open subset we mean an open set whose complement 
has codimension $\geqslant 2$.
When $d\gg0$
the Quot scheme $\mc Q$ is a local complete intersection. 
This follows easily using \cite[Proposition 2.2.8]{HL}
and is the content of Lemma \ref{local complete intersection}.
Using dimension bounds from \cite{PR} we show that the locus 
of singular points in $\mc Q$ has large codimension. 
These are used to prove Theorem \ref{det flat}.
To compute the Picard group, we first show that the locus 
of quotients $[E\longrightarrow F]$ with $F$ stable 
is a very large open subset. Let $\mc Q^s$ denote this locus.
We consider a map $\mc Q^s\longrightarrow M^s$,
to a moduli space of stable bundles of rank $k$ and suitable 
degree, see \eqref{diagram of principal bundles}. 
After base change by a principal ${\rm PGL}(N)$-bundle,
the domain becomes a very large open subset of a projective bundle 
associated to a vector bundle. From this we compute the Picard 
group of $\mc Q^s$ in terms of the Picard group of $M^s$. 
The assertions about $\mc Q_L$ follow in a 
similar manner using the assertions about 
$\mc Q$ and the flat map ${\rm det}:\mc Q\longrightarrow {\rm Pic}^d(C)$.

\section{Preliminaries}
For a locally closed subset $Z\subset X$
we shall refer to ${\rm dim}(X)-{\rm dim}(Z)$ 
as the codimension of $Z$ in $X$. 
For a morphism $f:X\longrightarrow Y$ and a closed point $y\in Y$ 
we denote by $X_y$ the fiber over $Y$. 
\begin{lemma}\label{fiber-dimension}
	Let $f:X\longrightarrow Y$ be a dominant morphism of integral schemes of 
	finite type over a field $k$. Let $U\subset X$ be an open subset such that 
	nonempty fibers of $f\vert_U$ have constant dimension. Let $Z:=X\setminus U$. 
	\begin{enumerate}
		\item If ${\rm dim}(X)-{\rm dim}(Z)> {\rm dim}(Y)$ 
		then the dimension of nonempty fibers of $f$ is constant.
		\item Let $t_0\geqslant 0$ be an integer and assume 
		${\rm dim}(X)-{\rm dim}(Z)> {\rm dim}(Y)+t_0$. 
		Let $y\in Y$ be a closed point such that $Z_y$ is nonempty. 
		Then ${\rm dim}(X_y)-{\rm dim}(Z_y)>t_0$.
	\end{enumerate}
\end{lemma}
\begin{proof}
	Let $y\in Y$ be a closed point such that $X_y$ is nonempty. 
	Note that $X_y=U_y\bigsqcup Z_y$. 
	If $U_y$ is empty then 
	$${\rm dim}(Z)\geqslant {\rm dim}(Z_y)={\rm dim}(X_y)\geqslant {\rm dim}(X)-{\rm dim}(Y)\,.$$
	This contradicts the hypothesis that ${\rm dim}(X)-{\rm dim}(Z)> {\rm dim}(Y)$.
	Thus, $U_y$ is nonempty. Since $f\vert_U$ has constant fiber dimension, 
	it follows that ${\rm dim}(U_y)={\rm dim}(U)-{\rm dim}(Y)$, 
	see \cite[Chapter 2, Exercise 3.22(b), (c)]{Ha}. Since $X$ is integral, it follows that 
	${\rm dim}(U_y)={\rm dim}(X)-{\rm dim}(Y)$.
	As ${\rm dim}(X)-{\rm dim}(Z_y)\geqslant {\rm dim}(X)-{\rm dim}(Z)>{\rm dim}(Y)$ it follows
	that ${\rm dim}(Z_y)<{\rm dim}(X)-{\rm dim}(Y)$. 
	It follows that 
	$${\rm dim}(X_y)={\rm max}\{{\rm dim}(U_y),{\rm dim}(Z_y)\}={\rm dim}(X)-{\rm dim}(Y)\,.$$
	This proves (1).
	
	Let $y\in Y$ be a closed point such that $Z_y$ is nonempty. Then $X_y$
	is nonempty and so by the previous part we get that  
	${\rm dim}(X_y)={\rm dim}(X)-{\rm dim}(Y)$. As
	${\rm dim}(X)-{\rm dim}(Z_y)\geqslant {\rm dim}(X)-{\rm dim}(Z)>{\rm dim}(Y) + t_0$ it follows
	that ${\rm dim}(Z_y)<{\rm dim}(X)-{\rm dim}(Y)-t_0={\rm dim}(X_y)-t_0$. 
	This proves (2) and completes the proof of the Lemma.
\end{proof}

\begin{lemma}\label{fiber-dimension-1}
	Let $f:X\longrightarrow Y$ be a morphism of irreducible schemes of 
	finite type over a field $k$ which is surjective on closed points. 
	Let $Y'\subset Y$ be a closed subset. Then 
	${\rm dim}(X)-{\rm dim}(f^{-1}(Y'))\leqslant {\rm dim}(Y)-{\rm dim}(Y')$.
\end{lemma}
\begin{proof}
	Since 
	it suffices to consider reduced schemes, we look at the map 
	$f_{\rm red}:X_{\rm red}\longrightarrow Y_{\rm red}$.
	Thus, we may assume that $X$ and $Y$ are integral schemes.
	Let $Y''\subset Y'$ be an irreducible component such that 
	${\rm dim}(Y'')={\rm dim}(Y')$. Let $Z''$ be an irreducible 
	component of $f^{-1}(Y'')$ which surjects onto $Y''$. 
	By \cite[Chapter 2, Exercise 3.22(a)]{Ha} we have 
	${\rm dim}(X)-{\rm dim}(Z'')\leqslant {\rm dim}(Y)-{\rm dim}(Y'')$.
	As $Z''\subset f^{-1}(Y')$ it follows that 
	$${\rm dim}(X)-{\rm dim}(f^{-1}(Y'))\leqslant{\rm dim}(X)-{\rm dim}(Z'')\leqslant 
		{\rm dim}(Y)-{\rm dim}(Y'')={\rm dim}(Y)-{\rm dim}(Y')\,.$$
	This completes the proof of the Lemma.
\end{proof}

Recall the space $\mc Q$ from \eqref{def Q}.
Let 
\begin{equation}
	p_1:C\times \mc Q\longrightarrow C\qquad p_2:C\times \mc Q\longrightarrow \mc Q
\end{equation}
denote the projections. Let
\begin{equation}
	0\longrightarrow \mc K\longrightarrow p_1^*E\longrightarrow \mc F\longrightarrow 0
\end{equation}
denote the universal quotient on $C\times \mc Q$. The sheaf 
$\mc K$ is locally free and so $p_1^*{\rm det}(E)\otimes (\wedge^{r-k}\mc K)^{-1}$ is a line 
bundle on $C\times \mc Q$ which is flat over $\mc Q$. Using this 
we define the determinant map as 
\begin{equation}\label{det}
	{\rm det}:\mc Q\longrightarrow {\rm Pic}^dC\,,
\end{equation}
which has the following pointwise description.
Let $[q:E\longrightarrow F]\in \mc Q$ be a closed point. We denote the kernel
of $q$ by $K$, so that there 
is a short exact sequence 
\begin{equation}\label{q}
	0\longrightarrow K\longrightarrow E\stackrel{q}{\longrightarrow} F\longrightarrow 0\,.
\end{equation}
Then 
$${\rm det}(q):={\rm det}(E)\otimes {\rm det}(K)^{-1}= {\rm det}(F)\,.$$
Next we describe the differential of this map ${\rm det}$. 

\begin{lemma}\label{description differential}
	The differential of the map det \eqref{det}
	at the point $q$ is the composite 
	$${\rm Hom}(K,F)\stackrel{-\delta}{\longrightarrow} 
		{\rm Ext}^1(F,F)\stackrel{tr}{\longrightarrow} H^1(C,\mc O_C)\,,$$
	where the first map is obtained by applying ${\rm Hom}(-,F)$ 
	to \eqref{q} and the second map is the trace.
\end{lemma}

\begin{proof}
	Let $p_C:C\times {\rm Spec}\,(\mb C[\epsilon]/(\epsilon^2))\longrightarrow C$ 
	denote the projection. Let 
	$\iota:C\hookrightarrow C\times {\rm Spec}\,(\mb C[\epsilon]/(\epsilon^2))$
	denote the reduced subscheme. 
	
	Given a vector $v\in {\rm Hom}(K,F)$ it corresponds to an
	element in the Zariski tangent space at $q\in \mc Q$,
	and so it corresponds to a short exact sequence on 
	$C\times {\rm Spec}\,(\mb C[\epsilon]/(\epsilon^2))$
	$$0\longrightarrow \t K\longrightarrow p_C^*E\longrightarrow \t F\longrightarrow 0\,,$$
	whose restriction to $C$ gives the sequence \eqref{q}.
	Moreover, $\t F$ is flat over ${\rm Spec}\,(\mb C[\epsilon]/(\epsilon^2))$.
	Consider the line bundle ${\rm det}(\t F)$ on $C\times {\rm Spec}\,(\mb C[\epsilon]/(\epsilon^2))$.
	Tensoring this line bundle with the short exact sequence 
	\begin{equation}\label{ses k epsilon}
		0\longrightarrow (\epsilon) \longrightarrow \mb C[\epsilon]/(\epsilon^2)
			\longrightarrow \mb C\longrightarrow 0
	\end{equation}
	gives the short exact sequence of sheaves on $C\times {\rm Spec}\,(\mb C[\epsilon]/(\epsilon^2))$
	\begin{equation}\label{ddet_q}
		0\longrightarrow \iota_*{\rm det}(F)\longrightarrow {\rm det}(\t F)\longrightarrow 
			\iota_*{\rm det}(F)\longrightarrow 0\,.
	\end{equation}
	Using the definition of the differential of the map ${\rm det}$ it is clear
	that 
	\begin{equation}
	d\,{\rm det}_q(v)=\text{extension class of \eqref{ddet_q}}\in H^1(C,\mc O_C)\,.
	\end{equation}
	
	Tensoring \eqref{ses k epsilon} with $\t F$ gives 
	a short exact sequence 
	\begin{equation}\label{ext class t F}
		0\longrightarrow \iota_*F\longrightarrow \t F\longrightarrow \iota_*F\longrightarrow 0\,.
	\end{equation}
	One checks easily, using the discussion before \cite[Lemma 2.2.6]{HL}, 
	that the above extension, and in particular the sheaf $\t F$,
	is obtained by taking the pushout of the sequence 
	\eqref{q} along the map $-v$. That is,
	the extension class of \eqref{ext class t F} in ${\rm Ext}^1(F,F)$ is precisely
	$-\delta(v)$.
	
	For a coherent sheaf $G$, consider the 
	trace map 
	$tr:{\rm Ext}^1(G,G)\longrightarrow H^1(C,\mc O_C)$.
	An element $v\in {\rm Ext}^1(G,G)$ corresponds to a short exact sequence 
	$$0\longrightarrow G \longrightarrow \t G\longrightarrow G\longrightarrow 0\,,$$
	on $C\times {\rm Spec}\,(\mb C[\epsilon]/(\epsilon^2))$ such that 
	$\t G$ is flat over ${\rm Spec}\,(\mb C[\epsilon]/(\epsilon^2))$.
	The image $tr(v)$ in $H^1(C,\mc O_C)$ corresponds 
	to the extension class obtained by tensoring \eqref{ses k epsilon}
	with the line bundle ${\rm det}(\t G)$ on $C\times {\rm Spec}\,(\mb C[\epsilon]/(\epsilon^2))$.
	When $G$ is locally free this can be seen using a Cech description, 
	for example, see \cite{Nitsure}.
	The general case reduces to the locally free case using 
	the discussion in \cite[\S10.1.2]{HL}. In particular, 
	we can apply this discussion by taking $G=F$. We get that 
	$tr(-\delta(v))$ is the extension class obtained by tensoring 
	${\rm det}(\t F)$ in \eqref{ext class t F} with \eqref{ses k epsilon}.
	But note that we obtained \eqref{ddet_q} also by tensoring ${\rm det}(\t F)$
	with \eqref{ses k epsilon}. This proves that 
	$$d\,{\rm det}_q(v)=tr(-\delta(v))$$
	and completes the proof of the Lemma. We also refer the reader 
	to \cite[Theorem 4.5.3]{HL}, where a similar result is proved
	for the moduli of stable bundles. 
\end{proof}

\section{Quot Schemes ${\rm Quot}_{C/\C}(E,r-1,d)$}
Recall that for a sheaf $G$ on $C$ we define $\mu_{\rm min}(G)$
as 
$${\rm min}\{\mu(F)\,\vert\, \text{ $F$ is a quotient of $G$ of positive rank}\, \}\,.$$
In this section we describe the Quot scheme ${\rm Quot}_{C/\C}(E,r-1,d)$
which parametrizes quotients of $E$ of rank $(r-1)$ and degree $d>2g-2+e-\mu_{\rm min}(E)$.
Let 
\begin{equation}\label{def rho_i}
	\rho_1:C\times {\rm Pic}^{e-d}(C)\longrightarrow C\,,\quad \quad 
		\rho_2:C\times {\rm Pic}^{e-d}(C)\longrightarrow {\rm Pic}^{e-d}(C)
\end{equation}
be the projections.
Let $\mc L$ be a Poincare bundle on $C\times {\rm Pic}^{e-d}(C)$. Define 
$$\mc E:=\rho_{2*}[\rho_{1}^*E\otimes \mc L^{\vee}]\,.$$ 
\begin{lemma}\label{lemma rank mc E}
	Assume $d>2g-2+e-\mu_{\rm min}(E)$. Then $\mc E$ is a vector bundle on ${\rm Pic}^{e-d}(C)$ of rank $rd-(r-1)e-r(g-1)$.   
\end{lemma}
\begin{proof}
	Let $K_C$ denote the canonical bundle of $C$.
	For any $L\in {\rm Pic}^{e-d}(C)$, we claim
	$$H^1(C,E\otimes L^{\vee})=H^0(C,E^{\vee}\otimes L\otimes K_C)^{\vee}=0\,.$$
	This is because a nonzero section of $H^0(C,E^{\vee}\otimes L\otimes K_C)$
	corresponds to a nonzero map
	$E\longrightarrow L\otimes K_C$ which cannot exist since by assumption 
	$\mu_{\rm min}(E)> {\rm deg}(L\otimes K_C)=e-d+2g-2$. Therefore by Grauert's theorem $\mc E$ is a vector bundle of rank $h^0(C,E\otimes L^{\vee})$ which by Riemann-Roch is $rd-(r-1)e-r(g-1)$.
\end{proof}
Let $\pi:\mb P(\mc E^{\vee})\longrightarrow {\rm Pic}^{e-d}(C)$ 
be the projective bundle associated to $\mc E^{\vee}$. 
Here we use the notation in \cite{Ha}, that is, for a vector space $V$, $\mb P(V)$ denotes 
the space of hyperplanes in $V$. Thus, $\mb P(V^\vee)$ denotes 
the space of lines in $V$. 
Recall that we have the map 
$$\mc Q_{C/\C}(E,r-1,d)\longrightarrow {\rm Pic}^{e-d}(C)$$ 
which sends a quotient $[E\longrightarrow F\longrightarrow 0]$ to its kernel.
\begin{theorem}\label{thm when k=r-1}
	Assume $d>2g-2+e-\mu_{\rm min}(E)$.
	We have an isomorphism of schemes over ${\rm Pic}^{e-d}(C)$
	$$\mb P(\mc E^{\vee})\xrightarrow{\sim} \mc Q_{C/\C}(E,r-1,d)\,.$$
	In particular, under the above assumption on $d$, the space 
	$\mc Q_{C/\C}(E,r-1,d)$ is smooth.
\end{theorem}
\begin{proof}
	Let 
	$$\sigma_1:C\times \mb P(\mc E^\vee)\longrightarrow C\,,\quad \quad 
		\sigma_2:C\times \mb P(\mc E^\vee)\longrightarrow \mb P(\mc E^\vee)$$
	be the projections.
	We define the map $\mb P(\mc E^{\vee})\longrightarrow \mc Q_{C/\C}(E,r-1,d)$ by producing a quotient on $C\times \mb P(\mc E^{\vee})$. 
	
	Recall the maps $\rho_i$ from \eqref{def rho_i}. 
	By adjunction we have a natural map on $C\times {\rm Pic}^{e-d}(C)$
	$$\rho^*_2\mc E \otimes \mc L \longrightarrow \rho^*_1E\,.$$
	Pulling this morphism back to $C\times \mb P(\mc E^{\vee})$ we get a map
	$$({\rm Id}_C\times \pi)^*[\rho_2^*\mc E\otimes \mc L]
	= (\pi\circ \sigma_2)^*\mc E \otimes ({\rm Id}_C\times \pi)^*\mc L \longrightarrow \sigma_1^*E\,.$$
	We also have the morphism of sheaves on $\mb P(\mc E^{\vee})$
	$$\mc O(-1)\hookrightarrow \pi^*\mc E\,.$$
	Pulling this back to $C\times \mb P(\mc E^{\vee})$ we 
	get a composed map of sheaves on $C\times \mb P(\mc E^{\vee})$
	\begin{equation}\label{eqn universal kernel}
		\sigma_2^*\mc O(-1)\otimes ({\rm Id}_C\times \pi)^*\mc L \longrightarrow 
			(\pi \circ \sigma_2)^*\mc E \otimes ({\rm Id}_C\times \pi)^*\mc L 
			\longrightarrow \sigma_1^* E\,.
	\end{equation}
	As $\sigma_2^*\mc O(-1)\otimes ({\rm Id}_C\times \pi)^*\mc L$
	is a line bundle and $C\times \mb P(\mc E^{\vee})$ is smooth,
	it easily follows that \eqref{eqn universal kernel} is an inclusion
	as it is nonzero. By the previous lemma, a point $x\in \mb P(\mc E^{\vee})$ 
	corresponds to a pair $(L,\phi:L\longrightarrow E)$ where 
	$L$ is a line bundle of degree $e-d$ and $\phi$ is a 
	nonzero homomorphism of sheaves, up to scalar multiplication. 
	The inclusion 
	(\ref{eqn universal kernel}) restricted to 
	$C\times x$ is nothing but the nonzero homomorphism $\phi$. 
	Therefore we get that the cokernel of 
	$(\ref{eqn universal kernel})$, which we denote $\mc F$, is flat over 
	$\mb P(\mc E^{\vee})$, and the restriction $\mc F\vert_{C\times x}$ 
	has rank $r-1$ and degree $d$. Thus, $\mc F$ defines a map 
	$\phi:\mb P(\mc E^{\vee})\longrightarrow \mc Q_{C/\C}(E,r-1,d)$. 
	It is easily checked that this map is bijective on closed points. 
	
	Let point $x=[E\longrightarrow F\longrightarrow 0]$ be a point in $\mc Q_{C/\C}(E,r-1,d)$. 
	Let $L$ be the kernel. Then we have an exact sequence
	$${\rm Ext}^1(L,L)\longrightarrow {\rm Ext}^1(L,E)\longrightarrow 
		{\rm Ext}^1(L,F)\longrightarrow 0\,.$$
	From the proof of Lemma \ref{lemma rank mc E} it follows that 
	${\rm Ext}^1(L,E)=0$. Hence ${\rm Ext}^1(L,F)=0$. Therefore 
	$\mc Q_{C/\C}(E,r-1,d)$ is smooth at 
	$x$ \cite[Proposition 2.2.8]{HL}. As $\phi$ is bijective on 
	closed points, it follows it is an isomorphism.
\end{proof}

\begin{corollary}\label{cor when k=r-1}
	Assume $d>2g-2+e-\mu_{\rm min}(E)$. Then $\mc Q_{C/\C}(E,r-1,d)$
	is a smooth projective variety and 
	${\rm Pic}(\mc Q_{C/\C}(E,r-1,d))\cong {\rm Pic}({\rm Pic}^0(C))\times \mb Z$.
\end{corollary}

When $E$ is the trivial bundle, Theorem \ref{thm when k=r-1} is proved in \cite[Corollary 4.23]{BDW}.

\section{The good locus for torsion free quotients}

The following Lemma is an easy consequence of \cite[Lemma 6.1]{PR}.

\begin{lemma}\label{lemma to define good locus}
	Let $k$ be an integer.
	There is a number $\mu_0(E,k)$, which depends only on $E$ and $k$, 
	such that for all torsion
	free sheaves $F$ with ${\rm rk}(F)\leqslant k$ and 
	$\mu_{\rm min}(F)\geqslant \mu_0(E,k)$ we have 
	$H^1(E^\vee\otimes F)=0$.
\end{lemma}
\begin{proof}
	When $F$ is stable and ${\rm rk}(F)\leqslant k$, it 
	follows using \cite[Lemma 6.1]{PR}, that there is 
	$\mu_0(E,k)$ such that if ${\rm rk}(F)\leqslant k$
	and $\mu(F)\geqslant \mu_0(E,k)$ then $H^1(E^\vee\otimes F)=0$. 
	
	Next let $F$ be semistable (see Remark following \cite[Lemma 6.1]{PR}). 
	Take a Jordan-Holder filtration 
	for $F$ and let $G$ be a graded piece of this filtration. 
	As ${\rm rk}(G)\leqslant k$ and $\mu(G)=\mu(F)\geqslant \mu_0(E,k)$
	it follows from the stable case that $H^1(E^\vee\otimes G)=0$. 
	From this it easily follows that if $F$ is semistable, ${\rm rk}(F)\leqslant k$ and 
	$\mu(F)\geqslant \mu_0(E,k)$ then $H^1(E^\vee\otimes F)=0$. 
	
	Now let $F$ be a 
	locally free sheaf with ${\rm rk}(F)\leqslant k$ and let 
	$$0=F_0\subsetneq F_1\subsetneq\ldots\subsetneq F_l=F$$
	be its Harder-Narasimhan filtration. Each graded piece is 
	semistable with slope 
	$$\mu(F_i/F_{i-1})\geqslant \mu(F_l/F_{l-1})=\mu_{\rm min}(F)\,.$$
	Thus, if $\mu_{\rm min}(F)\geqslant \mu_0(E,k)$ then 
	$\mu_{\rm min}(F_i/F_{i-1})\geqslant \mu_0(E,k)$
	and so from the semistable case it follows that 
	$H^1(E^\vee\otimes (F_i/F_{i-1}))=0$.
	Again it follows
	that $H^1(E^\vee\otimes F)=0$. This proves the lemma.
\end{proof}

Let $G$ be a locally free sheaf on $C$ and let $k$ be an integer. Define 
\begin{equation}\label{def d_k(E)}
d_k(G):={\rm min}\{d\,\vert\, \text{$\exists$ quotient $G\longrightarrow F$ with deg($F$)=$d$, rk$(F)=k$}\}.
\end{equation}

\begin{remark}\label{dimension quot}
We recall some results from \cite{PR} (see \cite[Lemma 6.1, Proposition 6.1, Theorem 6.4]{PR} 
and the remarks following these). There is an integer $\alpha(E,k)$ 
such that when $d\geqslant \alpha(E,k)$, the following three assertions hold:
\begin{enumerate}
	\item If $F$ is a stable bundle of rank $k$ 
	and degree $d$, then $E^\vee\otimes F$ is globally 
	generated.
	\item $\mc Q$ is irreducible and generically smooth 
	of dimension $rd-ek-k(r-k)(g-1)$.
	\item For the general 
	quotient $E\longrightarrow F$, with $F$ having rank $k$
	and degree $d$, we have the sheaf $F$ is torsion free 
	and stable.\hfill \qedsymbol
\end{enumerate}
\end{remark}

\begin{definition}
	Let $a,b$ be integers. Let ${\rm Quot}_{C/\C}(E,a,b)$
	be the Quot scheme parametrizing quotients of $E$ of rank $a$ and degree $b$.
	For a locally closed subset $A\subset{\rm Quot}_{C/\C}(E,a,b)$  
	define the following locally closed subsets of $A$.
	\begin{align}\label{good locus}
		A_{\rm g}&:=\{[E\longrightarrow F]\in A\,\vert\,  H^1(E^\vee\otimes F)=0 \}\\
		A_{\rm b}&:=A\setminus A_{\rm g}\nonumber \\
		A^{\rm tf}&:=\{[E\longrightarrow F]\in A\,\vert\,  \text{$F$ is torsion free} \}\nonumber \\
		A^{\rm tf}_{\rm g}&:=A^{\rm tf}\cap A_{\rm g}\nonumber \\
		A^{\rm tf}_{\rm b}&:=A^{\rm tf}\cap A_{\rm b}\nonumber
	\end{align}
	In particular, we get subsets $\mc Q^{\rm tf}_{\rm g}$, $\mc Q^{\rm tf}_{\rm b}$.\\

\noindent
For integers $0<k''<k<r$ define constants 
\begin{align}\label{define constants}
	C_1(E,k,k'')&:=k''(r-k'')-d_{k''}(E)r+(k-k'')(r-k)-d_{k}(E)(r-k'')\\
	C_2(E,k,k'')&:=-ek-k(r-k)(g-1)-C_1(E,k,k'')\nonumber \\
	C_3(E,k)&:=\underset{k''<k}{\rm min}\{C_2(E,k,k'')\}\,.\nonumber 
\end{align}

\noindent
Let $t_0$ be a positive integer. Define
\begin{equation}\label{beta}
	\beta(E,k,t_0):={\rm max}\{(r-1)\mu_0(E,r-1),r^2\mu_0(E,r-1)+t_0-C_3(E,k),\alpha(E,k),1\}\,.
\end{equation}
\end{definition}

\begin{remark}
	From the definition it is clear that $\beta(E,k,t_0)\geqslant \alpha(E,k)$
	for all integers $t_0\geqslant 1$,
	if $t_1\geqslant t_0\geqslant 1$
	then $\beta(E,k,t_1)\geqslant \beta(E,k,t_0)$ 
	and $\beta(E,k,t_0)\geqslant 1$ for all positive integers $t_0$.
	To define the constants $C_1,C_2,C_3$ we need that 
	$r\geqslant 3$. Note that if $r=2$, then the only possible value
	for $k$ is $1$, which equals $r-1$. This case has been dealt with in 
	the previous section. Thus, from now on we may assume that $r\geqslant 3$. 
	These constants will play a role while computing 
	dimensions of some subsets of $\mc Q_{C/\mb C}(E,k,d)$. 
	We emphasize that these constants are independent of $d$.  
\end{remark}

\begin{lemma}\label{lemma 1 codim}
	Fix positive integers $t_0$ and $k$ such that $k<r$. 
	Let $d\geqslant \beta(E,k,t_0)$. 
	Let $S$ be an irreducible component of $\mc Q^{\rm tf}_{\rm b}$.
	Then ${\rm dim}(\mc Q)-{\rm dim}(S)>t_0$ and 
	so also ${\rm dim}(\mc Q)-{\rm dim}(\mc Q^{\rm tf}_{\rm b})>t_0$.
\end{lemma}
\begin{proof}
	We give $S$ the reduced subscheme structure so that $S$ is an integral 
	scheme.
	Let $q\in S$ be a closed point corresponding to a quotient 
	$E\longrightarrow F$. If $F$ is semistable, then using 
	$d\geqslant  \beta(E,k,t_0)\geqslant (r-1)\mu_0(E,r-1)$ (note that 
	as $\beta(E,k,t_0)>0$ we have $d>0$) we get 
	$$\mu(F)=\mu_{\rm min}(F)=\frac{d}{k}\geqslant \frac{d}{r-1}\geqslant \mu_0(E,r-1)\,.$$
	It follows from Lemma \ref{lemma to define good locus} that 
	$q\in \mc Q^{\rm tf}_{\rm g}$, which is a contradiction as 
	$q\in \mc Q^{\rm tf}_{\rm b}$. Thus,
	$F$ is not semistable.
	
	Let $p_1:C\times S\longrightarrow C$ denote the projection. 
	Consider the pullback of the universal quotient from $C\times \mc Q$ 
	to $C\times S$ and denote it 
	$$p_1^*E\longrightarrow \mc F\,.$$
	From \cite[Theorem 2.3.2]{HL} (existence of relative Harder-Narasimhan filtration) 
	it follows that there is a dense 
	open subset $U\subset S$ and a filtration 
	$$0=\mc F_0\subsetneq\mc F_1\subsetneq\ldots\subsetneq\mc F_l=\mc F$$
	such that $\mc F_i/\mc F_{i-1}$ is flat over $U$ and for each 
	closed point $u\in U$, the sheaf $\mc F_{i,u}/\mc F_{i-1,u}$ is semistable.
	Consider the quotient $p_1^*E\longrightarrow \mc F_l\longrightarrow \mc F_l/\mc F_{l-1}$.
	Denote the kernel by $\mc S$ so that we have an exact sequence 
	$$0\longrightarrow \mc S\longrightarrow p_1^*E\longrightarrow \mc F_l/\mc F_{l-1}\longrightarrow 0$$
	on $C\times U$. Let us denote 
	$$\mc F'':=\mc F_l/\mc F_{l-1}\,,\qquad \mc F':=\mc F_{l-1}\,.$$
	With this notation we have a commutative diagram of short 
	exact sequences on $C\times U$
	\begin{equation}\label{e1}
	\xymatrix{
		0\ar[r] & \mc S\ar[r]\ar[d] & p_1^*E \ar[r]\ar[d] & \mc F''\ar[r]\ar@{=}[d] &0\phantom{\,.}\\
		0\ar[r] & \mc F'\ar[r] & \mc F\ar[r]& \mc F''\ar[r] &0\,.
	}
	\end{equation}
	In particular, we observe that the map $E\longrightarrow \mc F_u$ can be obtained as the pushout
	of the short exact sequence 
	$0\longrightarrow \mc S_u\longrightarrow E\longrightarrow \mc F''_u\longrightarrow 0$ along the map
	$\mc S_u\longrightarrow \mc F'_u$.
	
	For a closed point $u\in U$ define 
	$$k'':={\rm rk}(\mc F''_u)\,,\qquad d'':={\rm deg}(\mc F''_u)\,.$$
	Then 
	$${\rm rk}(\mc F'_u)=k-k''\,,\qquad {\rm deg}(\mc F'_u)=d-d''\,.$$
	The top row of \eqref{e1} defines a map 
	$$\theta:U\longrightarrow {\rm Quot}_{C/\C}(E,k'',d'')\,.$$
	For ease of notation let us denote $A:={\rm Quot}_{C/\C}(E,k'',d'')$.
	Let $\mc S_1$ denote the universal kernel bundle on 
	$C\times A$. Then $({\rm Id}_C\times\theta)^*\mc S_1=\mc S$. The left vertical
	arrow of \eqref{e1} defines a map to the relative Quot scheme
	\[\xymatrix{
		U\ar[drr]_\theta\ar[rr]^<<<<<<<<<{\t \theta}&& 
			{\rm Quot}_{C\times A/A}(\mc S_1,k-k'',d-d'')\ar[d]^\pi\\
		&& A
	}
	\]
	We claim that the map $\t\theta$ is injective on closed points. 
	Let $u_1,u_2\in U$ be such that $\t\theta(u_1)=\t\theta(u_2)$.
	Then $\theta(u_1)=\theta(u_2)$. It follows that the quotients 
	$E\longrightarrow \mc F''_{u_1}$ and $E\longrightarrow \mc F''_{u_2}$
	are the same, that is, $\mc S_{u_1}=\mc S_{u_2}$. 
	Since $\t\theta(u_1)=\t\theta(u_2)$ it follows that the quotients
	$\mc S_{u_1}\longrightarrow \mc F'_{u_1}$ and $\mc S_{u_2}\longrightarrow \mc F'_{u_2}$
	are the same.  
	We observed after \eqref{e1}, that the quotient $E\longrightarrow \mc F_{u_i}$ 
	is obtained as the pushout of the short exact sequence 
	$0\longrightarrow \mc S_{u_i}\longrightarrow E\longrightarrow \mc F''_{u_i}\longrightarrow 0$ 
	along the map
	$\mc S_{u_i}\longrightarrow \mc F'_{u_i}$. From this it follows
	that the quotients $E\longrightarrow \mc F_{u_i}$ are the same. Thus, the
	map $\t\theta$ is injective on 
	closed points. 
	
	Let us compute the dimension of 
	${\rm Quot}_{C\times A/A}(\mc S_1,k-k'',d-d'')$.
	Consider a quotient $[E\longrightarrow F'']$ which corresponds 
	to a closed point in $A$. Let $S_{F''}$ denote the kernel.
	It has rank $r-k''$.
	The fiber of $\pi$ over $[E\longrightarrow F'']$ is the Quot scheme
	${\rm Quot}_{C/\C}(S_{F''},k-k'',d-d'')$. Recall from \eqref{def d_k(E)} 
	the integer $d_{k-k''}(S_{F''})$,
	which is the smallest possible degree among all 
	quotients of $S_{F''}$ of rank $k-k''$. Thus, if the fiber
	is nonempty then we have that 
	$$d-d''\geqslant d_{k-k''}(S_{F''})\,.$$
	By \cite[Theorem 4.1]{PR} it follows that, if the fiber is 
	nonempty then 
	\begin{align}\label{e2}
		{\rm dim}({\rm Quot}_{C/\C}(S_{F''},k-k'',d-d'')) & \leqslant 
		(k-k'')(r-k)+\\
		&(d-d''-d_{k-k''}(S_{F''}))(r-k'')\,.\nonumber
	\end{align}
	We will find a lower bound for $d_{k-k''}(S_{F''})$. Let $S_{F''}\longrightarrow G$
	be a quotient such that ${\rm deg}(G)=d_{k-k''}(S_{F''})$.
	Then we can form the pushout $\t G$ which sits in the 
	following commutative diagram
	\begin{equation*}
		\xymatrix{
			0\ar[r] & S_{F''}\ar[r]\ar[d] & E \ar[r]\ar[d] & F''\ar[r]\ar@{=}[d] &0\phantom{\,.}\\
			0\ar[r] & G\ar[r] & \t G\ar[r]& F''\ar[r] &0\,.
		}
	\end{equation*}
	Since $\t G$ is a quotient 
	of $E$ of rank $k$, it follows that 
	$${\rm deg}(\t G)=d''+d_{k-k''}(S_{F''})\geqslant d_k(E)\,.$$
	This shows that $d_{k-k''}(S_{F''})\geqslant d_k(E)-d''$.
	Combining this with \eqref{e2} yields
	\begin{align}\label{e3}
		{\rm dim}({\rm Quot}_{C/\C}(S_{F''},k-k'',d-d''))  \leqslant 
		(k-k'')(r-k)+
		(d-d_{k}(E))(r-k'')\,.
	\end{align}
	Again using \cite[Theorem 4.1]{PR} it follows that 
	\begin{equation}\label{e4}
		{\rm dim}(A)={\rm dim}({\rm Quot}_{C/\C}(E,k'',d''))\leqslant k''(r-k'')+(d''-d_{k''}(E))r\,.
	\end{equation}
	Combining \eqref{e3} and \eqref{e4}, using \eqref{define constants} and injectivity 
	of $\t\theta$ we get that 
	\begin{align*}
		{\rm dim}(U)&\leqslant {\rm Quot}_{C\times A/A}(\mc S_1,k-k'',d-d'')\\
			& \leqslant k''(r-k'')+(d''-d_{k''}(E))r+(k-k'')(r-k)+(d-d_{k}(E))(r-k'')\\
			&=C_1(E,k,k'') + d(r-k'')+d''r
	\end{align*}
	From this, Remark \ref{dimension quot}(2) and \eqref{define constants} it follows that 
	\begin{align}\label{e5}
		{\rm dim}(\mc Q)-{\rm dim}(U)\geqslant C_2(E,k,k'')+dk''-d''r\,.
	\end{align}
	We claim that $C_2(E,k,k'')+dk''-d''r>t_0$. If not, then we have
	$$\frac{C_2(E,k,k'')+dk''-t_0}{r}\leqslant d''\,.$$ 
	But this yields
	\begin{equation}\label{e6}
		\frac{C_3(E,k)+d-t_0}{r^2}\leqslant \frac{C_2(E,k,k'')+d-t_0}{r^2}
		< \frac{C_2(E,k,k'')+dk''-t_0}{rk''}\leqslant \frac{d''}{k''}\,.
	\end{equation}
	Let $u\in U$ be a closed point. Then $\mu_{\rm min}(\mc F_u)=d''/k''$.
	By the assumption on $d$ we have that 
	$$d\geqslant \beta(E,k,t_0)\geqslant r^2\mu_0(E,r-1)+t_0-C_3(E,k)\,.$$ 
	Using this and \eqref{e6} gives
	$$\mu_0(E,r-1)\leqslant \frac{C_3(E,k)+d-t_0}{r^2} <\frac{d''}{k''}=\mu_{\rm min}(\mc F_u)\,.$$
	It follows from Lemma \ref{lemma to define good locus} that 
	$H^1(E^\vee\otimes \mc F_u)=0$, that is, $u\in \mc Q^{\rm tf}_{\rm g}$.
	But this is a contradiction as $U\subset \mc Q^{\rm tf}_{\rm b}$.
	Thus, it follows from \eqref{e5} that 
	$${\rm dim}(\mc Q)-{\rm dim}(S)\geqslant C_2(E,k,k'')+dk''-d''r>t_0\,.$$
	This completes the proof of the Lemma.
\end{proof}

\section{Locus of quotients which are not torsion free}

For a sheaf $F$, denote the torsion subsheaf of $F$ by ${\rm Tor}(F)$.
For an integer $i\geqslant 1$ define the locally closed subset 
\begin{equation}\label{def Z_i}
	Z_i:=\{[q:E\longrightarrow F]\in \mc Q\,\vert\, {\rm deg}({\rm Tor}(F))=i\}\,.
\end{equation}
We now estimate the dimension of $Z_i$ and $(Z_i)_{\rm b}$ (recall the definition of $(Z_i)_b$ from
\eqref{good locus}). 

\begin{lemma}\label{lemma 2 codim}
	With notation as above we have 
	\begin{enumerate}
		\item Assume that $d-i\geqslant \alpha(E,k)$ (see Remark \ref{dimension quot}).
		Then $Z_i$ is irreducible and 
		${\rm dim}(Z_i)= {\rm dim}(\mc Q)-ki$.
		Moreover, $\bar Z_i\supset \bigcup_{j\geqslant i}Z_j$.  
		\item Let $t_1$ be a positive integer. If $d-i\geqslant \beta(E,k,t_1)$ 
		(see \eqref{beta} for definition of $\beta$)
		then ${\rm dim}(Z_i)-{\rm dim}((Z_i)_{\rm b})>t_1$.
		\item If $d-i\geqslant \beta(E,k,t_1)$ then 
		${\rm dim}(\mc Q)-{\rm dim}((Z_i)_{\rm b})>t_1+ki$.
	\end{enumerate}
	
\end{lemma}
\begin{proof}
	Consider the Quot scheme ${\rm Quot}_{C/\C}(E,k,d-i)$. 
	For ease of notation we denote $A={\rm Quot}_{C/\C}(E,k,d-i)$.
	Let 
	$$0\longrightarrow \ms S\longrightarrow p_1^*E\longrightarrow \ms F\longrightarrow 0$$ 
	be the universal quotient on $C\times A$.
	Consider the relative Quot scheme
	\begin{equation}\label{rel-quot}
		{\rm Quot}_{C\times A/A}(\ms S,0,i)\stackrel{\pi}{\longrightarrow} A\,.
	\end{equation} 
	There is a map 
	\begin{equation}\label{rel-quot to Q}
		{\rm Quot}_{C\times A/A}(\ms S,0,i)\stackrel{\pi'}{\longrightarrow} \mc Q
	\end{equation}
	whose image consists of precisely those quotients $[E\longrightarrow F]$ for which
	${\rm deg}({\rm Tor}(F))\geqslant i$. 
	Recall the locus $A^{\rm tf}$ from \eqref{good locus}.
	One checks easily that 
	\begin{equation}\label{lem2codime1}
		\pi'^{-1}(Z_i)=\pi^{-1}(A^{\rm tf})\,.
	\end{equation}
	In fact, one easily checks that 
	$\pi': \pi^{-1}(A^{\rm tf})\longrightarrow Z_i$ 
	is a bijection on points and so they have the same dimension.
	As $d-i\geqslant \alpha(E,k)$,
	by Remark \ref{dimension quot}(2),
	it follows that $A$ is irreducible of dimension 
	$${\rm dim}(A)=r(d-i)-ek-k(r-k)(g-1)\,.$$
	By Remark \ref{dimension quot}(3),
	it follows that $A^{\rm tf}$ is a dense open subset of $A$.
	If $[E\longrightarrow F]\in A$ is a quotient, let $S_F$
	denote the kernel. The fiber of $\pi$ 
	over this point is the Quot scheme ${\rm Quot}_{C/\C}(S_F,0,i)$,
	which is irreducible and has dimension $(r-k)i$.
	From this it follows that ${\rm Quot}_{C\times A/A}(\ms S,0,i)$
	is irreducible of dimension ${\rm dim}(\mc Q)-ki$.
	Thus, the open set $\pi^{-1}(A^{\rm tf})$ also has the same
	dimension and is irreducible. As this open subset dominates $Z_i$, 
	the claim about the irreducibility and dimension of $Z_i$ follows.
	We have already observed that the image of $\pi'$
	is the locus $\bigcup_{j\geqslant i}Z_j$. As 
	$\pi^{-1}(A^{\rm tf})$ is dense in ${\rm Quot}_{C\times A/A}(\ms S,0,i)$,
	the proof of (1) is complete.
	
	To prove the second assertion, note that 
	$$H^1(E^\vee\otimes F)=H^1(E^\vee\otimes (F/{\rm Tor}(F)))\,.$$
	One checks easily that 
	\begin{equation}\label{lem2codime2}
		\pi'^{-1}((Z_i)_{\rm b})=\pi^{-1}(A^{\rm tf}_{\rm b})\,.
	\end{equation}
	As $\pi$ has constant fiber dimension, we see 
	$${\rm dim}(A^{\rm tf})-{\rm dim}(A^{\rm tf}_{\rm b})=
	{\rm dim}(\pi^{-1}(A^{\rm tf}))-{\rm dim}(\pi^{-1}(A^{\rm tf}_{\rm b}))\,.$$
	By applying Lemma \ref{fiber-dimension-1} to the map $\pi'$, and using 
	\eqref{lem2codime1} and \eqref{lem2codime2}, we get  
	\begin{align*}
		{\rm dim}(A^{\rm tf})-{\rm dim}(A^{\rm tf}_{\rm b})&=
		{\rm dim}(\pi^{-1}(A^{\rm tf}))-{\rm dim}(\pi^{-1}(A^{\rm tf}_{\rm b}))\\
		&={\rm dim}(\pi'^{-1}(Z_i))-{\rm dim}(\pi'^{-1}((Z_i)_{\rm b}))
		\leqslant {\rm dim}(Z_i)-{\rm dim}((Z_i)_{\rm b})\,.
	\end{align*}
	As $d-i\geqslant \beta(E,k,t_1)\geqslant \alpha(E,k)$ it follows 
	from Remark \ref{dimension quot}(2) and (3) that ${\rm Quot}_{C/\C}(E,k,d-i)$
	is irreducible and so 
	${\rm dim}({\rm Quot}_{C/\C}(E,k,d-i))={\rm dim}({\rm Quot}_{C/\C}(E,k,d-i)^{\rm tf})$.
	By Lemma \ref{lemma 1 codim} it follows that 
	$${\rm dim}(A^{\rm tf})-{\rm dim}(A^{\rm tf}_{\rm b})=
	{\rm dim}({\rm Quot}_{C/\C}(E,k,d-i)^{\rm tf})-
		{\rm dim}({\rm Quot}_{C/\C}(E,k,d-i)^{\rm tf}_{\rm b})>t_1\,.$$
	This proves that ${\rm dim}(Z_i)-{\rm dim}((Z_i)_{\rm b})>t_1$.
	This proves (2).
	
	Assertion (3) of the Lemma follows easily using the first two. 
\end{proof}

\section{Flatness of det}

We begin by showing that when $d\gg0$, $\mc Q$ is a local complete intersection. 
This result seems well known to experts (see \cite[Theorem 1.6]{BDW}
and the paragraph following it); however, we include it as we could not 
find a precise reference. 
\begin{lemma}\label{local complete intersection}
	Let $d\geqslant \alpha(E,k)$. Then $\mc Q$ is a local complete intersection scheme.
	In particular, it is Cohen-Macaulay.
\end{lemma}
\begin{proof}
	By Remark \ref{dimension quot}(2), $\mc Q$ is irreducible and so ${\rm dim}_{q}(\mc Q)$ 
	is independent of the closed point $q\in \mc Q$. 
	Let $\mc F$ denote the universal quotient 
	and let $\mc K$ denote the universal kernel on $C\times \mc Q$.
	For 
	a closed point $q\in \mc Q$ we shall denote the restrictions 
	of these sheaves to $C\times q$ by $\mc K_q$ and $\mc F_q$. 
	The sheaf 
	$\mc K$ is locally free on $C\times \mc Q$. It follows 
	that $\mc K^\vee\otimes \mc F$ is flat over $\mc Q$, and so the 
	Euler characteristic of $\mc K^\vee_q\otimes \mc F_q$ is constant, call it $\chi$. 
	As $\mc Q^{\rm tf}_{\rm g}$
	is nonempty, let $q\in \mc Q^{\rm tf}_{\rm g}$ be a closed point.
	As $h^1(\mc K^\vee_q\otimes \mc F_q)=0$, it follows from 
	\cite[Proposition 2.2.8]{HL} that 
	$${\rm dim}_q(\mc Q)=h^0(\mc K^\vee_q\otimes \mc F_q)=
	h^0(\mc K^\vee_q\otimes \mc F_q)-h^1(\mc K^\vee_q\otimes \mc F_q)=\chi\,.$$ 
	Let $t\in \mc Q$ be a closed point. 
	We already observed that ${\rm dim}_{t}(\mc Q)$ 
	is independent of the closed point $t\in \mc Q$ and so is equal to $\chi$. 
	It follows that for all closed points $t\in \mc Q$ we have 
	$${\rm dim}_t(\mc Q)=\chi=h^0(\mc K^\vee_t\otimes \mc F_t)-h^1(\mc K^\vee_t\otimes \mc F_t)\,.$$
	By \cite[Proposition 2.2.8]{HL} it follows that the space 
	$\mc Q$ is a local complete intersection at any closed point and 
	so is also Cohen-Macaulay.
\end{proof}

\begin{lemma}\label{compact lemma codimension}
	Fix a positive integer $t_0$. Let $i_0$ be the smallest 
	integer such that $ki_0>g(C)+t_0$. If $d\geqslant \beta(E,k,g(C)+t_0)+i_0$
	then ${\rm dim}(\mc Q)-{\rm dim}(\mc Q_{\rm b})>g(C)+t_0$.
\end{lemma}
\begin{proof}
	First observe that we can write
	$$\mc Q=\mc Q^{\rm tf}\sqcup \bigsqcup_{i\geqslant 1}Z_i\,.$$
	Only finitely many indices $i$ appear.
	In fact, $i$ can be at most $d-d_{k}(E)$, see \eqref{def d_k(E)}.
	In view of this we get 
	$$\mc Q_{\rm b}=\mc Q^{\rm tf}_{\rm b}\sqcup \bigsqcup_{i\geqslant 1}(Z_i)_{\rm b}\,.$$
	By Lemma \ref{lemma 1 codim}, since $d\geqslant \beta(E,k,g(C)+t_0)$ we have  
	$${\rm dim}(\mc Q)-{\rm dim}(\mc Q^{\rm tf}_{\rm b})>g(C)+t_0\,.$$
	If $1\leqslant i\leqslant i_0$ then  
	$d-i\geqslant d-i_0\geqslant \beta(E,k,g(C)+t_0)$,
	and so by Lemma \ref{lemma 2 codim}(3) we get 
	$${\rm dim}(\mc Q)-{\rm dim}((Z_i)_{\rm b})>g(C)+t_0+ki\,.$$
	By Lemma \ref{lemma 2 codim}(1) we also get that 
	$\bar Z_{i_0}\supset \bigcup_{j\geqslant i_0}Z_j$.
	For $j\geqslant i_0$, 
	$${\rm dim}((Z_j)_{\rm b})\leqslant{\rm dim}(Z_j)\leqslant {\rm dim}(Z_{i_0})= {\rm dim}(\mc Q)-ki_0\,.$$
	This shows that for $j\geqslant i_0$ we have 
	$${\rm dim}(\mc Q)-{\rm dim}((Z_j)_{\rm b})\geqslant ki_0>g(C)+t_0\,.$$
	Combining these shows that ${\rm dim}(\mc Q)-{\rm dim}(\mc Q_{\rm b})>g(C)+t_0$.
	This completes the proof of the Lemma.
\end{proof}

\begin{theorem}\label{det flat}
	Recall the map {\rm det} defined in \eqref{det}. 
	\begin{enumerate}
	\item 
	Let $n_0$ be the smallest integer such that $kn_0>g(C)+1$.
	Let $d\geqslant \beta(E,k,g(C)+1)+n_0$.  
	Then ${\rm det}:\mc Q\longrightarrow {\rm Pic}^d(C)$ is a flat map.
	Further, $\mc Q$ is an integral and normal variety. 
	\item Let $n_1$ be the smallest integer such that $kn_1>g(C)+3$.
	Let $d\geqslant \beta(E,k,g(C)+3)+n_1$. Then $\mc Q$ is locally factorial.
	\end{enumerate}
\end{theorem}
\begin{proof}
Let $q\in \mc Q$ be a closed point and let
$K$ denote the kernel of the quotient $q$. Then we have a 
short exact sequence 
$$0\longrightarrow K\longrightarrow E\longrightarrow F\longrightarrow 0\,.$$
Applying ${\rm Hom}(-,F)$ and using Lemma \ref{description differential}
we get the following 
diagram, in which the top row is exact.
\begin{equation}\label{eq-smoothness-det}
	\xymatrix{
		{\rm Hom}(K,F)\ar[r]\ar[dr]_{d({\rm det})_q}& {\rm Ext}^1(F,F)\ar[d]^{tr}\ar[r]
			&{\rm Ext}^1(E,F)\ar[r] & {\rm Ext}^1(K,F)\ar[r]&0\\
		&	H^1(C,\mc O_C)
	}
\end{equation}
If $H^1(E^\vee\otimes F)=0$ then we make the following two observations. 
First observe that it follows that $H^1(K^\vee\otimes F)=0$,
which shows that $\mc Q_{\rm g}$ is contained in the smooth
locus of $\mc Q$, by \cite[Proposition 2.2.8]{HL}. 
Second observe that the map ${\rm Hom}(K,F)\longrightarrow {\rm Ext}^1(F,F)$
will be surjective.
As ${\rm Ext}^1(F,F)\longrightarrow H^1(C,\mc O_C)$ is surjective, it follows
that if $H^1(E^\vee\otimes F)=0$ then the diagonal map in the 
above diagram is surjective. However, the diagonal map is precisely
the differential of det at the point $q$. As ${\mc Q}_{\rm g}$ and 
${\rm Pic}^d(C)$ are smooth, it follows that the restriction of 
det to ${\mc Q}_{\rm g}$ is a smooth morphism and so also flat
and dominant.

Assume $d\geqslant \beta(E,k,g(C)+1)+n_0$. 
Applying Lemma \ref{compact lemma codimension} we get
$${\rm dim}(\mc Q)-{\rm dim}(\mc Q_{\rm b})>g(C)+1\,.$$
We observed in Lemma \ref{local complete intersection} 
that $\mc Q$ is a Cohen-Macaulay scheme and so it satisfies Serre's condition
$S_2$. The open subset $\mc Q_{\rm g}$ 
is smooth. As $\mc Q_{\rm b}=\mc Q\setminus \mc Q_{\rm g}$, it follows
that $\mc Q$ satisfies Serre's condition $R_1$. Thus, 
$\mc Q$ is an integral and normal variety. 

In view of Lemma \ref{local complete intersection} and  
\cite[Theorem 23.1]{Mat} or
\cite[\href{https://stacks.math.columbia.edu/tag/00R4}{Tag 00R4}]{Stk},
to prove the first assertion of the theorem, it suffices to show that the fibers 
of det have constant dimension. 
Applying Lemma \ref{fiber-dimension}(1), by taking $U$ to be 
the open subset $\mc Q_{\rm g}$, we get that det is flat.
This proves (1). 

Now we prove (2). 
Assume $d\geqslant \beta(E,k,g(C)+3)+n_1$.
Applying Lemma \ref{compact lemma codimension} we get
$${\rm dim}(\mc Q)-{\rm dim}(\mc Q_{\rm b})>g(C)+3\,.$$
This implies that the singular locus has codimension 4 or more.
Now we use a result of Grothendieck which states 
that if $R$ is a local ring that is a complete intersection 
in which the singular locus has
codimension 4 or more, then $R$ is a UFD. We refer the reader to 
\cite{sga2}, \cite{Call}, \cite[Theorem 1.4]{HT}.
This implies that $\mc Q$ is locally factorial. 
The proof of the theorem is now complete.
\end{proof}

\section{Locus of stable quotients and Picard group of $\mc Q$}\label{section stable quotients}

\subsection{}\label{recall-Bhosle}
In this section we will be using two Quot schemes. Thus, it 
is worth recalling that $\mc Q$ denotes the Quot scheme ${\rm Quot}_{C/\mb C}(E,k,d)$.
We begin by explaining a result from \cite{Bhosle}
that we need. Assume one of the following two holds
\begin{itemize}
	\item $k\geqslant 2$ and $g(C)\geqslant 3$, or
	\item $k\geqslant 3$ and $g(C)=2$.
\end{itemize}
Let $d\geqslant \alpha(E,k)$.
Fix a closed point $P\in C$. For a closed point $q\in \mc Q$,
let $[E\stackrel{q}{\longrightarrow}\mc F_q]$
denote the quotient corresponding to this closed point. 
We may choose $n\gg 0$ such that for all 
$q\in \mc Q^{\rm tf}_{\rm g}$ we have $H^1(C,\mc F_q(nP))=0$
and $\mc F_q(nP)$ is globally generated. As $d\geqslant \alpha(E,k)$, by Remark 
\ref{dimension quot}, it follows that $\mc Q^{\rm tf}_{\rm g}$ is irreducible, 
and so $h^0(C,\mc F_q(nP))$ is independent of $q$. 
Let 
\begin{equation}\label{def N}
	N:=h^0(C,\mc F_q(nP))
\end{equation}
and consider the Quot scheme ${\rm Quot}_{C/\C}(\mc O_C^{\oplus N},k,d+kn)$.
Let $\mc G'$ denote the universal quotient on 
$C\times {\rm Quot}_{C/\C}(\mc O_C^{\oplus N},k,d+kn)$.
Let $R\subset {\rm Quot}_{C/\C}(\mc O_C^{\oplus N},k,d+kn)$ 
be the open subset containing closed points 
$[x:\mc O_C^{\oplus N}\longrightarrow\mc G'_x]$ such that 
$\mc G'_x$ is torsion free, $H^1(C,\mc G'_x)=0$ and the quotient
map $\mc O_C^{\oplus N}\longrightarrow \mc G'_x$ induces an isomorphism 
$\C^N\stackrel{\sim}{\longrightarrow} H^0(C,\mc G'_x)$.
This is the space $R$ in \cite[page 246, Proposition 1.2]{Bhosle},
see \cite[page 246, Notation 1.1]{Bhosle}.
The space $R$ is a smooth equidimensional scheme. 
Let $R^s$ (respectively, $R^{ss}$) denote the open subset 
of $R$ consisting of closed points $x$ for which $\mc G'_x$
is stable (respectively, semistable). 
In \cite[page 246, Proposition 1.2]{Bhosle} it is proved
that ${\rm dim}(R)-{\rm dim}(R\setminus R^s)\geqslant 2$.

Let 
\begin{align*}
	p_1:C\times {\rm Quot}_{C/\C}(\mc O_C^{\oplus N},k,d+kn)&\longrightarrow C\\
	p_2:C\times {\rm Quot}_{C/\C}(\mc O_C^{\oplus N},k,d+kn)&\longrightarrow 
	{\rm Quot}_{C/\C}(\mc O_C^{\oplus N},k,d+kn)
\end{align*}
denote the projections.
Let 
$$\mc G:=\mc G'\otimes p_1^*(\mc O_C(-nP))\,.$$
Let $R'\subset R$ be the open subset containing closed points 
$x$ for which $H^1(C,E^\vee\otimes \mc G_x)=0$. 
By Cohomology and Base change theorem it follows
that $p_{2*}(p_1^*E^\vee\otimes \mc G)$ is locally 
free on $R'$.
The fiber over a point $x\in R'$ is isomorphic to 
the vector space ${\rm Hom}(E,\mc G_x)$.
Consider the projective bundle 
\begin{equation}\label{def Theta}
	\mb P(p_{2*}(p_1^*E^\vee\otimes \mc G)^\vee)\stackrel{\Theta}{\longrightarrow} R'\,.
\end{equation}
The fiber of $\Theta$ over a point $x\in R'$ is the space
of lines in the vector space ${\rm Hom}(E,\mc G_x)$.
For ease of notation we denote 
$\mb P(p_{2*}(p_1^*E^\vee\otimes \mc G)^\vee)$ by $\mb P$.
Denote the projection maps from $C\times \mb P$ by
$$p'_1:C\times \mb P\longrightarrow C\,,\qquad  p'_2:C\times \mb P\longrightarrow \mb P\,.$$
Consider the following Cartesian square
\[\xymatrix{
	C\times \mb P\ar[r]^{\t \Theta}\ar[d]_{p'_2} & C\times R'\ar[d]^{p_2}\\
	\mb P\ar[r]^\Theta & R'
}
\]
Let $\mc O(1)$ denote the tautological line bundle 
on $\mb P$. Then we have a map of sheaves on $C\times \mb P$
\begin{equation}\label{map of sheaves C times P-1}
	p'^*_1 E\longrightarrow \t \Theta^*\mc G\otimes p'^*_2\mc O(1)\,.
\end{equation}
A closed point $v\in \mb P$ corresponds to the closed point $\Theta(v)\in R'$
and a line spanned by some $w_v \in {\rm Hom}(E,\mc G_{\Theta(v)})$. The restriction
of \eqref{map of sheaves C times P-1} to $C\times v$ gives the map $w_v:E\longrightarrow \mc G_{\Theta(v)}$.
Let $\mb U\subset \mb P$ denote the open subset 
parametrizing points $v$ such that $w_v$ is surjective.
On $C\times \mb{U}$ we have a surjection 
\begin{equation}
	p'^*_1 E\longrightarrow \t \Theta^*\mc G\otimes p'^*_2\mc O(1)\,.
\end{equation}
This defines a morphism 
\begin{equation}\label{def Psi}
	\Psi:\mb{U}\longrightarrow \mc Q^{\rm tf}_{\rm g}\,.
\end{equation}
\begin{lemma}
	$\Psi$ is surjective on closed points. 
\end{lemma}
\begin{proof}
	Let $[q:E\longrightarrow \mc F_q]\in \mc Q^{\rm tf}_{\rm g}$ be a closed point. 
		By our choice of $n$ and $N$ (see \eqref{def N}), we have that 
		$\mc F_q(nP)$ is globally generated and $N=h^0(C,\mc F_q(nP))$. Therefore, by choosing
		a basis for $H^0(C,\mc F_q(nP))$ 
		we get a surjection
		$[\mc O_C^N\longrightarrow \mc F_q(nP)]$.
		Now it follows easily that $\Psi$ is surjective 
		on closed points.
\end{proof}

Now further assume $d \geqslant {\rm max}\{\alpha(E,k),k\mu_0(E,k)\}$. By 
Lemma \ref{lemma to define good locus} we have 
$H^1(C,E^\vee\otimes \mc G_x)=0$ for $x\in R^s$.
Thus, we have inclusions of open sets $R^s\subset R'\subset R$. 
Let $\mb P^s\subset \mb P$ denote the inverse image of $R^s$ 
under the map $\Theta$. Similarly, let $\mb U^s\subset \mb U$ denote the inverse image of $R^s$ 
under the restriction of $\Theta$ to $\mb U$.
Let 
\begin{equation}\label{def Q^s}
	\mc Q^{\rm s}:=\{[E\longrightarrow F]\in \mc Q\,\vert\, \text{$F$ is stable}\}.
\end{equation}
As $d \geqslant k\mu_0(E,k)$, by 
Lemma \ref{lemma to define good locus} we have 
$H^1(C,E^\vee\otimes F)=0$ for $[E\longrightarrow F]\in \mc Q^s$.
It follows that $\mc Q^s\subset \mc Q^{\rm tf}_{\rm g}$.
It is easily checked that 
\begin{equation}\label{inverse image Q^s}
	\Psi^{-1}(\mc Q^s)=\mb U^s\,.
\end{equation} 
The group ${\rm PGL}(N)$ acts freely on $\mb{P}^s$
and leaves the open subset $\mb U^s$ invariant.
Consider the trivial action of ${\rm PGL}(N)$ 
on $\mc Q^s$.
Then the restriction $\Psi:\mb U^s\longrightarrow \mc Q^s$ 
is ${\rm PGL}(N)$-equivariant.
It is clear that the restriction of the map $\Theta:\mb P^s\longrightarrow R^s$ 
is also ${\rm PGL}(N)$-equivariant.
Let $M^s_{k,d+kn}$ (respectively, $M_{k,d+kn}$) 
denote the moduli space of stable (respectively, semistable) 
bundles of rank $k$ and degree $d+kn$. Then $M^s_{k,d+kn}$
is the GIT quotient 
$$\psi: R^s{\longrightarrow}R^s/\!\!/{\rm PGL}(N)=M^s_{k,d+kn}\,.$$
Let $p_C:C\times \mc Q\longrightarrow C$ denote the projection
and let $p_C^*E\longrightarrow \mc F$ denote the universal quotient
on $C\times \mc Q$. The sheaf $p_C^*\mc O_C(nP)\otimes \mc F$
on $C\times \mc Q^s$ defines a morphism 
$\mc Q^s\stackrel{\theta}{\longrightarrow} M^s_{k,d+kn}$.
One easily checks that we have the following 
commutative diagram, in which all arrows are surjective
on closed points 
\begin{equation}\label{diagram of principal bundles}
	\xymatrix{
		\mb U^s\ar[rr]^\Psi\ar[d]_{\Theta_{\mb U^s}} && \mc Q^s\ar[d]^{\theta}\\
		R^s\ar[rr]^{\psi}&& M^s_{k,d+kn}\,.
	}
\end{equation}
The map $\psi$ is a principal ${\rm PGL}(N)$-bundle. For a closed 
point $x\in R^s$, the points in the fiber $\Theta^{-1}_{\mb U^s}(x)$ are in bijection
with the points in the fiber $\theta^{-1}(\psi(x))$. Here we use the stability 
of the quotient sheaf to assert that no two distinct points in the fiber $\Theta^{-1}_{\mb U^s}(x)$ 
map to the same point in the fiber $\theta^{-1}(\psi(x))$.
The natural map
from $\mb U^s$ to the Cartesian product of $\psi$ and $\theta$
is a bijective map of smooth varieties and hence an isomorphism.
This shows that the above diagram is Cartesian. 

In this section we shall compute the Picard group of $\mc Q$ when $d\gg0$.
As we saw in Theorem \ref{det flat}, $\mc Q$ is locally factorial
and so the Picard group is isomorphic to the divisor class group. 
Let $CH^1(\mc Q)$ denote the divisor class group of $\mc Q$.
We shall first show that $CH^1(\mc Q)\stackrel{\sim}{\longrightarrow}CH^1(\mc Q^s)$
and then use the diagram \eqref{diagram of principal bundles} to compute $CH^1(\mc Q^s)$.

In the following Lemma we shall use the fact that $\mb U$ is irreducible.
This is easily seen as follows. The moduli space $M^s_{k,k+dn}$
is an integral scheme. It easily follows that $R^s$ is irreducible
as $M^s_{k,k+dn}$ is the GIT quotient $R^s/\!\!/{\rm PGL}(N)$.  
By \cite[Proposition 1.2]{Bhosle} we have that 
${\rm dim}(R)-{\rm dim}(R\setminus R^s)\geqslant 2$.
As $R$ is equidimensional, it follows that $R$ is irreducible. 
As $R$ is smooth it follows that $R$ is an integral scheme
and so is $R'$. It follows that $\mb U$ is integral. 
\begin{lemma}\label{pic of Q and Q^s}
	Assume one of the following two holds
	\begin{itemize}
		\item $k\geqslant 2$ and $g(C)\geqslant 3$, or
		\item $k\geqslant 3$ and $g(C)=2$.
	\end{itemize}
	Also assume $d\geqslant {\rm max}\{\alpha(E,k)+1,k\mu_0(E,k),\beta(E,k,1) \}$. 
	Then the  map $CH^1(\mc Q){\longrightarrow}CH^1(\mc Q^s)$
	is an isomorphism. 
\end{lemma}
\begin{proof}
	Recall the definition of $Z_1$ from \eqref{def Z_i}
	and observe that $\mc Q^{\rm tf}=\mc Q\setminus \bar Z_1$. 
	Taking $i=1$ in Lemma \ref{lemma 2 codim}(1) we get  
	${\rm dim}(\mc Q)-{\rm dim}(\bar Z_1)\geqslant k$.
	Since $k\geqslant 2$, it follows that 
	$CH^1(\mc Q)=CH^1(\mc Q^{\rm tf})$.
	
	By Lemma \ref{lemma 1 codim} it follows that 
	$${\rm dim}(\mc Q^{\rm tf})-{\rm dim}(\mc Q^{\rm tf}_{\rm b})=
		{\rm dim}(\mc Q)-{\rm dim}(\mc Q^{\rm tf}_{\rm b})>1\,.$$
	Observe that $\mc Q^{\rm tf}_{\rm g}=\mc Q^{\rm tf}\setminus \mc Q^{\rm tf}_{\rm b}$.
	It follows that $CH^1(\mc Q^{\rm tf})=CH^1(\mc Q^{\rm tf}_{\rm g})$.
	
	We had observed earlier that $\mc Q^s\subset \mc Q^{\rm tf}_{\rm g}$.
	To prove the Lemma it suffices to show that 
	$${\rm dim}(\mc Q^{\rm tf}_{\rm g})-{\rm dim}(\mc Q^{\rm tf}_{\rm g}\setminus \mc Q^s)>1\,.$$
	We will now show this.

	As $d \geqslant k\mu_0(E,k)$, by 
	Lemma \ref{lemma to define good locus} we have 
	$H^1(C,E^\vee\otimes \mc G_x)=0$ for $x\in R^s$.
	We have already checked above, see \eqref{inverse image Q^s}, 
	that $\Psi^{-1}(\mc Q^s)=\mb U^s$.
	
	As the map $\Theta$
	is flat and $\mb U$ is integral, it follows using Lemma \ref{fiber-dimension-1} 
	(applied to the map $\Psi:\mb U\longrightarrow \mc Q^{\rm tf}_{\rm g}$)
	that 
	$$2\leqslant {\rm dim}(R')-{\rm dim}(R'\setminus R^s)=
		{\rm dim}(\mb U)-{\rm dim}(\mb U\setminus \mb U^s)\leqslant 
		{\rm dim}(\mc Q^{\rm tf}_{\rm g})-{\rm dim}(\mc Q^{\rm tf}_{\rm g}\setminus \mc Q^s)\,.$$
	This completes the proof of the Lemma. 
\end{proof}

\begin{lemma}\label{Picard group of U^s}
	Let $r-k\geqslant2$. Let $d\geqslant {\rm max}\{\alpha(E,k),k\mu_0(E,k)+k\}$. 
	The natural map $CH^1(\mb P^s)\longrightarrow CH^1(\mb U^s)$ is an isomorphism.
\end{lemma}
\begin{proof}
	It suffices to show that 
	${\rm dim}(\mb P^s)-{\rm dim}(\mb P^s\setminus\mb U^s)\geqslant 2$.
	Let $[x:\mc O_C^{\oplus N}\longrightarrow F]$ be a quotient corresponding to 
	a closed point $x\in R^s$. It suffices to show that 
	${\rm dim}(\Theta^{-1}(x))-{\rm dim}(\Theta^{-1}(x)\setminus\mb U^s)\geqslant 2$
	for every closed point $x\in R^s$.
	We now show this. 
	
	The space $\Theta^{-1}(x)$ is the space $\mb P({\rm Hom}(E,F)^\vee)$ 
	parametrizing lines in the vector 
	space ${\rm Hom}(E,F)$. 
	Let $c\in C$ be a closed point. 
	As $F$ is stable, note $\mu_{\rm min}(F(-c))=\mu(F)-1$.
	As $d\geqslant k\mu_0(E,k)+k$, it follows that 
	$$\mu_{\rm min}(F(-c))=\mu(F)-1=\frac{d-k}{k}\geqslant \mu_0(E,k)\,.$$ 
	Let $p_i$ denote the projections from $C\times C$. Let $\Delta$ denote 
	the diagonal in $C\times C$. Consider the short exact sequence 
	of sheaves on $C\times C$
	given by 
	$$0\longrightarrow p_1^*(E^\vee\otimes F)(-\Delta)\longrightarrow p_1^*(E^\vee\otimes F)\longrightarrow \Delta_*(E^\vee\otimes F)\longrightarrow 0\,.$$
	By Lemma 
	\ref{lemma to define good locus} we have $H^1(E^\vee\otimes F(-c))=0$.
	Applying $p_{2*}$ to the above, 
	we get that the sheaf  
	$$\mc V:=p_{2*}(p_1^*(E^\vee\otimes F)(-\Delta))\,,$$ 
	which is locally free on $C$,
	sits in a short exact sequence 
	$$0\longrightarrow \mc V\longrightarrow {\rm Hom}(E,F)\otimes \mc O_C\longrightarrow E^\vee\otimes F\longrightarrow 0\,.$$
	The restriction of the above sequence to a closed point $c\in C$ gives the 
	short exact sequence of vector spaces 
	\begin{equation}\label{ses-1}
	0\longrightarrow {\rm Hom}(E,F(-c))\longrightarrow {\rm Hom}(E,F)
		\longrightarrow {\rm Hom}(E\vert_c,F\vert_c)\longrightarrow 0\,.
	\end{equation}
	Consider the closed subset $\mb P(\mc V^\vee)\subset \mb P({\rm Hom}(E,F)^\vee)\times C$.
	Let $T\subset \mb P({\rm Hom}(E,F)^\vee)$ denote the image of $\mb P(\mc V^\vee)$ under the 
	projection map 
	$$\mb P({\rm Hom}(E,F)^\vee)\times C \longrightarrow \mb P({\rm Hom}(E,F)^\vee)\,.$$ 
	Then $T$ is a closed subset and set theoretically it is the union 
	$$T=\bigcup_{c\in C}\mb P({\rm Hom}(E,F(-c))^\vee)\,.$$
	As $r-k\geqslant 2$ we have $rk\geqslant (k+2)k>2$. Therefore,
	\begin{equation}\label{codimension T}
	{\rm dim}(\mb P({\rm Hom}(E,F)^\vee))-{\rm dim}(T)\geqslant 
		{\rm dim}(\mb P({\rm Hom}(E,F)^\vee))-{\rm dim}(\mb P(\mc V^\vee))=rk-1\geqslant2\,.
	\end{equation}

	Let $V$ denote the open set $\mb P({\rm Hom}(E,F)^\vee)\setminus T$. 
	Let $\mc O(1)$ denote the restriction of the 
	tautological bundle on $\mb P({\rm Hom}(E,F)^\vee)$ to $V$. Let $p_C$ denote the projection from
	$C\times V$ to $C$ and let $p_V$ denote the projection to $V$. Consider the canonical
	map of sheaves on $C\times V$
	\begin{equation}\label{canonical map of sheaves C times V}
	p_C^*(E\otimes F^\vee)\longrightarrow {\rm Hom}(E,F)^\vee \otimes \mc O_{C\times V}
		\longrightarrow p_V^*\mc O(1)\,.
	\end{equation}
	Let $\varphi\neq 0$ be an element in ${\rm Hom}(E,F)$ such that the line $[\varphi]$ it 
	defines is in $V$. The dual of equation \eqref{canonical map of sheaves C times V} restricted to
	$C\times [\varphi]$ is described as follows. This restriction maps 
	$$\mb C\longrightarrow \mb C[\varphi]\otimes \mc O_C\longrightarrow E^\vee\otimes F\,.$$
	The second map is precisely the global section corresponding to the map $\varphi$.
	For a point $c\in C$, the map 
	\eqref{canonical map of sheaves C times V} restricted to $(c,[\varphi])$
	is adjoint to the map $E\vert_c\stackrel{\varphi\vert_c}{\longrightarrow} F\vert_c$.
	As $[\varphi]\in V$, it follows that the map 
	$E\vert_c\stackrel{\varphi\vert_c}{\longrightarrow} F\vert_c$
	is nonzero, and so it follows that the restriction 
	of \eqref{canonical map of sheaves C times V} 
	to $(c,[\varphi])$ is nonzero, that is, 
	$E\vert_c\otimes F\vert_c^\vee\longrightarrow \mb C$
	is nonzero and hence surjective. This proves that 
	the map \eqref{canonical map of sheaves C times V}
	is surjective. This defines a map 
	$C\times V\stackrel{\kappa}{\longrightarrow} \mb P(E\otimes F^\vee)$
	which sits in a commutative diagram
	\[\xymatrix{
		C\times V\ar[r]^<<<<<\kappa\ar[dr] & \mb P(E\otimes F^\vee)\ar[d]^\pi\\
		& C
	}
	\]
	The restriction of the map $\kappa$ over a point $c\in C$ is the composite map below,
	where the second arrow is obtained using \eqref{ses-1}
	$$V\longrightarrow\mb P({\rm Hom}(E,F)^\vee)\setminus \mb P({\rm Hom}(E,F(-c))^\vee)
		\longrightarrow \mb P({\rm Hom}(E\vert_c,F\vert_c)^\vee)\,.$$
	The second arrow is a surjective flat map and the first arrow is an open immersion. 
	It follows that the composite is a flat map and hence 
	has constant fiber dimension. It follows that the map $\kappa$ has constant fiber dimension,
	and so using \cite[Theorem 23.1]{Mat} or
	\cite[\href{https://stacks.math.columbia.edu/tag/00R4}{Tag 00R4}]{Stk}
	we see that $\kappa$ is a flat map. Consider the canonical map 
	$$\pi^*E\longrightarrow \pi^*F\otimes \mc O_{\mb P(E\otimes F^\vee)}(1)$$
	on $\mb P(E\otimes F^\vee)$
	and let $Z$ denote the support of the cokernel. The set 
	$Z\cap \pi^{-1}(c)$ is precisely the locus of non-surjective 
	maps in $\mb P(E\vert_c\otimes F\vert_c^\vee)$. By \cite[Chapter II, \S2, page 67]{ACGH1}
	we have that the codimension of $Z\cap \pi^{-1}(c)$ in $\mb P(E\vert_c\otimes F\vert_c^\vee)$
	is $r-k+1$. It follows that the codimension of $Z$
	in $\mb P(E\otimes F^\vee)$ is $r-k+1$. It follows that 
	the codimension of $\kappa^{-1}(Z)$ in $C\times V$ is $r-k+1$
	and the codimension of $p_V(\kappa^{-1}(Z))$ in $V$ is at least $r-k\geqslant 2$.
	The set $V\setminus p_V(\kappa^{-1}(Z))$ is precisely the locus 
	of points in $\mb P({\rm Hom}(E,F)^\vee)$ corresponding to 
	maps which are surjective. The locus 
	of points in $\mb P({\rm Hom}(E,F)^\vee)$ corresponding to
	non-surjective maps $E\longrightarrow F$ is
	the set $T\bigcup p_V(\kappa^{-1}(Z))$, which has codimension at least 2. This proves that 
	${\rm dim}(\Theta^{-1}(x))-{\rm dim}(\Theta^{-1}(x)\setminus\mb U^s)\geqslant 2$,
	which completes the proof of the Lemma.
\end{proof}

\begin{remark}\label{remark k=1}
	The proof of Lemma \ref{Picard group of U^s} also shows the following. 
	Let $k=1$ and $r\geqslant 3$ so that $k\leqslant r-2$. 
	Let $d\geqslant {\rm max}\{\alpha(E,1),\mu_0(E,1)+1\}$. 
	Let $L$ be a line bundle on $C$ of degree $d$. Then the closed subset 
	in $\mb P({\rm Hom}(E,L)^\vee)$ consisting of non-surjective 
	maps has codimension $\geqslant 2$. 
\end{remark}

\begin{theorem}\label{picard group of Quot}
	Let $r-k\geqslant2$. Assume one of the following two holds
	\begin{itemize}
		\item $k\geqslant 2$ and $g(C)\geqslant 3$, or
		\item $k\geqslant 3$ and $g(C)=2$.
	\end{itemize}
	Let $n_1$ be the smallest integer such that $kn_1>g(C)+3$.
	Assume 
	$$d\geqslant {\rm max}\{\alpha(E,k)+1,k\mu_0(E,k)+k,\beta(E,k,g(C)+3)+n_1 \}\,.$$
	Then  
	$${\rm Pic}(\mc Q)\cong {\rm Pic}(M^s_{k,d+kn})\times \Z \cong 
		{\rm Pic}({\rm Pic}^0(C))\times \Z \times \Z\,.$$
\end{theorem}

\begin{proof}
	We saw in Theorem \ref{det flat} that $\mc Q$ is an integral variety 
	which is normal and locally factorial.
	So the Picard group is isomorphic to the divisor class group. 
	By Lemma \ref{pic of Q and Q^s} it is enough to show that 
	$${\rm Pic}(\mc Q^s)\cong {\rm Pic}(M^s_{k,d+kn})\times \Z\,.$$   
	Recall that we have the following diagram 
	\eqref{diagram of principal bundles}, which we checked is Cartesian:
	\begin{equation*}
	\xymatrix{
		\mb U^s\ar[rr]^\Psi\ar[d]_{\Theta_{\mb U^s}} && \mc Q^s\ar[d]^{\theta}\\
		R^s\ar[rr]^{\psi} && M^s_{k,d+kn}\,.
	}
	\end{equation*}
	Recall from \S\ref{recall-Bhosle} that we had fixed a closed point $P\in C$.
	Note that for any $[x:\mc O^N_C\longrightarrow F(nP)]\in R^s$, the fibre 
	$\Theta^{-1}_{\mb U^s}(x)\cong \theta^{-1}([F])$. In the proof of Lemma 
	\ref{Picard group of U^s} we proved that ${\rm dim}(\Theta^{-1}(x))-{\rm dim}(\Theta^{-1}(x)\setminus\mb U^s)\geqslant 2$
	for every closed point $x\in R^s$. It follows that
	$\Theta^{-1}_{\mb U^s}(x)=\Theta^{-1}(x)\cap \mb U^s$ is an open subset 
	of projective space (that is, $\Theta^{-1}(x)$)
	whose complement has codimension $\geqslant 2$. Thus,
        $$\mb Z={\rm Pic}(\Theta^{-1}(x))={\rm Pic}(\Theta^{-1}_{\mb U^s}(x))=
        	{\rm Pic}(\theta^{-1}([F]))\,.$$
        Therefore we have the restriction map 
        $${\rm res}: {\rm Pic}(\mc Q)\cong {\rm Pic}(\mc Q^s)\longrightarrow {\rm Pic}(\theta^{-1}([F]))
		\cong \mb Z\,.$$
	We claim this map is nontrivial. Let $\mc L$ be a very ample
	line bundle on $\mc Q$. If ${\rm res}(\mc L)$ were trivial, 
	it would follow that ${\rm res}(\mc L)$ is 
	trivial and very ample, which is a 
	contradiction as $\theta^{-1}([F])\cong \Theta^{-1}(x)$ 
	is an open subset of a projective
	space whose complement has codimension $\geqslant 2$. 
	Thus, the image of ${\rm res}$ is isomorphic to a copy of 
	$\mb Z$. We will show that  
	the kernel of ${\rm res}$ is isomorphic to ${\rm Pic}(M^s_{k,d+kn})$.
    
	Let $L\in {\rm Pic}(\mc Q^s)$ be such that ${\rm res}(L)$ 
	is trivial. We need to show that $L$ is 
	isomorphic to the pullback of some line bundle on 
	${\rm Pic}(M^s_{k,d+kn})$. Consider the pullback $\Psi^{*}L$. 
	Since $\Psi$ is ${\rm PGL}(N)$-invariant, this line bundle carries a 
	${\rm PGL}(N)$-linearization. By Lemma \ref{Picard group of U^s}, 
	the complement of $\mb U^s$ in $\mb P^s$ has 
	codimension $\geqslant 2$. 
	Therefore, both $L$ and this ${\rm PGL}(N)$-linearization 
	extend uniquely to $\mb P^s$. Let us denote this 
	extension of $\Psi^*L$ to $\mb P^s$ by $L'$ and 
	the linearization on ${\rm PGL}(N)\times \mb P^s$ by 
	$\alpha':m^*_{\mb P^s}L'\longrightarrow p_{\mb P^s}^*L$, 
	where $m_{\mb P^s}:{\rm PGL}(N)\times \mb P^s\longrightarrow \mb P^s$ 
	is the multiplication map and 
	$p_{\mb P^s}:{\rm PGL}(N)\times \mb P^s\longrightarrow \mb P^s$ 
	is the second projection. Since 
	$\Theta:\mb P^s\longrightarrow R^s$ is a 
	projective bundle,  $L'\cong \mc O(n)\otimes \Theta^*L''$ 
	for some $L''\in {\rm Pic}(R^s)$ and for some $n$. However, since the fibers 
	of $\Theta$ and $\theta$ are isomorphic, 
	the condition ${\rm res}(L)$ is trivial implies that 
	$n=0$, that is, $L'\cong \Theta^*L''$. Now note that 
	since the map $\mb P^s\longrightarrow R^s$ is 
	${\rm PGL}(N)$-equivariant we have a commutative diagram
	\[
	\begin{tikzcd}    
		{\rm PGL}(N)\times \mb P^s \ar[r,"m_{\mb P^s}"] \ar[d,"{\rm Id}\times \Theta"] & 
			\mb P^s \ar[d,"\Theta"] \\
		{\rm PGL}(N)\times R^s      \ar[r,"m_{R^s}"]            & R^s   
	\end{tikzcd}
	\]
	From this diagram it follows that we have an isomorphism of sheaves
	$$({\rm Id}\times \Theta)^*m^*_{R^s}L''\cong 
		m^*_{\mb P^s}\Theta^*L''\stackrel{\sim}{\longrightarrow} p_{\mb P^s}^* \Theta^*L''
		\cong ({\rm Id} \times \Theta)^*p_{R^s}^*L''\,.$$
	where the middle isomorphism is given by the linearization $\alpha'$. Since 
	${\rm Id}\times \Theta$ is a projective bundle, applying 
	$({\rm Id}\times \Theta)_*$ to this composition of isomorphisms 
	we get a linearization 
	$$\alpha'':m^*_{R^s}L''\stackrel{\sim}{\longrightarrow} p_{R^s}^*L''$$ 
	of $L''$
	such that $({\rm Id}\times \Theta)^*\alpha''=\alpha'$. Now recall that 
	the map $\psi$ is a principal ${\rm PGL}(N)$-bundle. By \cite[Theorem 4.2.14]{HL} 
	we get that there exists $L'''\in {\rm Pic}(M^s_{k,d+kn})$ such that 
	$\psi^*L'''\cong L''$ and the induced ${\rm PGL}(N)$ linearization is $\alpha''$. Therefore we get that 
	$$\Psi^* \theta^* L''' \cong \Theta^* \psi^*L''' \cong \Theta^* L''\cong L'\cong \Psi^*L$$
	and also the induced ${\rm PGL}(N)$-linearizations are also the same. 
	Since the diagram (\ref{diagram of principal bundles}) is Cartesian, the map 
	$\Psi$ is a principal ${\rm PGL}(N)$-bundle. Hence by \cite[Theorem 4.2.16]{HL} we get that 
	$\theta^*L'''\cong L$. This completes the proof of the first equality in the statement 
	of the Theorem. 
	The second equality follows from \cite[Theorem A, Theorem C]{DN} and from the 
	fact that 
	$${\rm dim}(M_{k,d+kn})-{\rm dim}(M_{k,d+kn}\setminus M^s_{k,d+kn})\geqslant 2\,.$$
	One way to see this inequality is to apply \cite[Proposition 1.2 (3)]{Bhosle} and 
	Lemma \ref{fiber-dimension-1} to the GIT quotient $R^{ss} \to M_{k,d+kn}$.
\end{proof}

\section{Fibers of ${\rm det}$}

Let $L$ be a line bundle on $C$ of degree $d$
and let $\mc Q_L$ denote the scheme theoretic fiber 
${\rm det}^{-1}(L)$. 
As a corollary of Theorem \ref{det flat}
we have the following Proposition.
\begin{proposition}\label{Q_L lci}
	Let $n_1$ be the smallest integer such that $kn_1>g(C)+3$.
	Let $d\geqslant \beta(E,k,g(C)+3)+n_1$.  
	Then $\mc Q_L$ is a local complete
	intersection scheme which is equidimensional, normal and 
	locally factorial.
\end{proposition}
\begin{proof}
	We use Theorem \ref{det flat} and Lemma \ref{local complete intersection}.
	As $\mc Q$ is a local complete intersection scheme, ${\rm Pic}^d(C)$
	is smooth and the map ${\rm det}$ is flat, it follows 
	using \cite[(1.9.2)]{Av} (see also \cite[Remark 2.3.5]{Bruns-Herzog}
	and \cite[\href{https://stacks.math.columbia.edu/tag/09Q2}{Tag 09Q2}]{Stk}) that $\mc Q_L$
	is a local complete intersection scheme and so also Cohen-Macaulay.
	As $\mc Q$ is irreducible, flatness of ${\rm det}$
	also implies that $\mc Q_L$ is equidimensional.

	We observed in the proof of Theorem \ref{det flat} 
	that the restriction of ${\rm det}$
	to the open subset $\mc Q_{\rm g}$ is a smooth morphism. 
	It follows that $\mc Q_L \cap \mc Q_{\rm g}$ is contained 
	in the smooth locus of $\mc Q_L$. The singular locus 
	of $\mc Q_L$ is thus contained in $\mc Q_L \cap \mc Q_{\rm b}$.
	As $d\geqslant \beta(E,k,g(C)+3)+n_1$,
	applying Lemma \ref{compact lemma codimension} we get
	$${\rm dim}(\mc Q)-{\rm dim}(\mc Q_{\rm b})>g(C)+3\,.$$
	By Lemma \ref{fiber-dimension}(2) it follows that 
	\begin{equation}\label{codimenion Q_{L,b}}
		{\rm dim}(\mc Q_L)-{\rm dim}(\mc Q_L \cap \mc Q_{\rm b})>3\,.
	\end{equation}
	It follows that the singular locus of $\mc Q_L$ has 
	codimension 4 or more. This proves that $\mc Q_L$ 
	is normal, that is, it is the disjoint union of finitely many 
	normal varieties, all of the same dimension. Using Grothendieck's theorem
	(see \cite[Theorem 1.4]{HT}) 
	it follows that $\mc Q_L$ is locally factorial.
\end{proof}

Next we want to find conditions under which $\mc Q_L$ 
becomes irreducible. We use the notation 
used in Lemma \ref{lemma 2 codim}.
In the proof of the next Lemma we will use the following fact.
Let $X\longrightarrow S$ be a projective morphism of schemes
with relative ample line bundle $\mc O(1)$. 
Let $\ms S$ be a coherent sheaf on $X$. Let $P(n)$ denote 
the constant polynomial defined by $P(n)=1$ for all $n$.
Then the relative Quot scheme ${\rm Quot}_{X/S}(\ms S,P)$ is isomorphic 
to $\mb P(\ms S)\longrightarrow X$.
\begin{lemma}\label{density Q^tf_L}
	Let $n_0$ be the smallest integer such that $kn_0>g(C)+1$.
	Let $n_1$ be the smallest integer such that $kn_1>g(C)+3$.
	Let 
	$$d\geqslant {\rm max}\{\beta(E,k,g(C)+1)+n_0+1, \beta(E,k,g(C)+3)+n_1\}\,.$$
	Then $\mc Q^{\rm tf}_L$ is dense in $\mc Q_L$.
\end{lemma}
\begin{proof}
	Recall the relative Quot scheme in equation \eqref{rel-quot}. 
	We are interested in the case $i=1$, that is, the relative 
	Quot scheme ${\rm Quot}_{C\times A/A}(\ms S,0,1)$, 
	where $A$ is the Quot scheme ${\rm Quot}_{C/\C}(E,k,d-1)$.
	For ease of notation we denote by $B$ the scheme ${\rm Quot}_{C\times A/A}(\ms S,0,1)$. 
	Recall the map $\pi:B\longrightarrow A$ from \eqref{rel-quot}. On $C\times B$
	we have a quotient 
	\begin{equation}\label{univ-quot-B}
		({\rm Id_C}\times \pi)^*\ms S\longrightarrow \ms T\,,
	\end{equation}
	such that $\ms T$ is flat over $B$. 
	Using $\ms T$ we get the determinant map
	$${\rm det}_B:B\longrightarrow {\rm Pic}^1(C)\,.$$
	This map has the following pointwise description. 
	A closed point $b\in B$ gives rise to the closed 
	point $\pi(b)\in A$, which corresponds to a short exact
	sequence on $C$
	$$0\longrightarrow S_F\longrightarrow E\longrightarrow F\longrightarrow 0\,,$$ where $F$ 
	is of rank $k$ and degree $d-1$ on $C$. The restriction 
	of the universal quotient \eqref{univ-quot-B}
	to the point $b$ is a torsion quotient on $C$
	$$S_F\longrightarrow M\,,$$
	such that ${\rm length}(M)=1$. Let $c={\rm Supp}(M)$.
	Then ${\rm det}_B(b)=\mc O_C(c)$. Consider the natural embedding
	(recall that $g(C)>0$)
	$\iota: C\hookrightarrow {\rm Pic}^1(C)$ given by $c\mapsto \mc O_C(c)$.
	It is clear that the image of $B$ is the image of $\iota$.
	Next we want to show that $B$ is an integral scheme. 
	
	As $d-1\geqslant \alpha(E,k)$, it follows from Lemma \ref{local complete intersection}
	that $A$ is a local complete intersection. 
	By Theorem \ref{det flat}(1)
	it follows that $A$ is integral. As $i=1$, using the fact stated before 
	this Lemma, it is easily 
	checked that $B$ is the 
	projective bundle $\mb P(\ms S)\longrightarrow C\times A$. It follows that 
	$B$ is integral and a local complete intersection
	and so Cohen-Macaulay. 
	As $B$ is integral, the map ${\rm det}_B$
	factors through the map $\iota$, that is, we have a commutative
	diagram 
	\[\xymatrix{
		B\ar[rr]^{{\rm det}_B}\ar[dr]_{{\rm det}_T}&&{\rm Pic}^1(C)\\
		&C\ar[ru]^{\iota}
	}
	\]
	Let ${\rm det}_A:A\longrightarrow {\rm Pic}^{d-1}(C)$ denote the determinant
	map for the Quot scheme $A$. This is flat due to Theorem \ref{det flat}(1). Consider the map
	\[
	\xymatrix{
		B\ar[rrr]_<<<<<<<<<<<<<<{({\rm det}_T,\,{\rm det}_A\circ \pi)}
		\ar@/^2.0pc/@[red][rrrr]^{{\rm det}_B}&&& 
			C\times {\rm Pic}^{d-1}(C)\ar[r]& {\rm Pic}^d(C)\,.
	}
	\]
	The second map is given by $(c,M)\mapsto M\otimes \mc O_C(c)$.
	It is easily checked that both maps have constant fiber dimension. 
	In view of \cite[Theorem 23.1]{Mat} it follows that both maps 
	are flat and so the composite ${\rm det}_B$ is also flat.
	Recall the map $\pi'$ from \eqref{rel-quot to Q}.
	It is clear that we have a commutative diagram
	\[\xymatrix{
		B\ar[r]^{\pi'}\ar[dr]_{{\rm det}_B} & \mc Q\ar[d]^{{\rm det}}\\
		& {\rm Pic}^d(C)\,.
	}
	\]
	Recall the definition of $Z_1$, see \eqref{def Z_i}.
	We saw in the proof of Lemma \ref{lemma 2 codim} that 
	$\pi'(B)=\bar Z_1$. Let 
	$$\bar Z_{1,L}:=\{[q:E\longrightarrow F]\in \bar Z_{1}\,\vert\, {\rm det}(F)=L\}\,.$$
	Let $B_L:={\rm det}_B^{-1}(L)$ denote the scheme theoretic
	fiber over $L$. Then it is clear that $\pi'(B_L)=\bar Z_{1,L}$.
	Thus, it follows that ${\rm dim}(\bar Z_{1,L})\leqslant {\rm dim}(B_L)$.
	In the proof of Lemma \ref{lemma 2 codim} (after equation \eqref{lem2codime1}) 
	we had remarked that 
	there is an open set $U\subset B$ such that $\pi'$ is injective 
	on points of $U$. It is easily checked that this open set $U$ meets all fibers $B_L$. 
	Thus, $\pi'$ is also injective on the subset $U\cap B_L$.
	Thus, it follows that ${\rm dim}(\bar Z_{1,L})\geqslant {\rm dim}(U\cap B_L)$.
	Since ${\rm det}_B$ is flat, the fibers are equidimensional and so
	it follows that every open set of $B_L$ has the same dimension as $B_L$.
	Combining these we get 
	\begin{equation}\label{codim of Q^{tf}}	
	{\rm dim}(\bar Z_{1,L})={\rm dim}(B_L)={\rm dim}(\mc Q)-k-g
		={\rm dim}(\mc Q_L)-k\,.
	\end{equation}	
	As $k\geqslant 1$, and all irreducible components of $\mc Q_L$
	have the same dimension, it follows that 
	$\mc Q_L\setminus \bar Z_{1,L}=\mc Q^{\rm tf}_L$
	is dense in $\mc Q_L$.
\end{proof}

The above Lemma implies that irreducibility of $\mc Q_L$ is equivalent to 
the irreducibility of the open subset $\mc Q_L^{\rm tf}$.
Let 
$$\mc Q^{\rm tf}_{{\rm g},L}:=\mc Q^{\rm tf}_{\rm g}\cap \mc Q_L\,.$$
Combining Proposition \ref{Q_L lci} and Lemma \ref{density Q^tf_L} we get the 
following. 

\begin{lemma}\label{density  Q^tf_g,L}
	Let $n_0$ be the smallest integer such that $kn_0>g(C)+1$.
	Let $n_1$ be the smallest integer such that $kn_1>g(C)+3$.
	Let 
	$$d\geqslant {\rm max}\{\beta(E,k,g(C)+1)+n_0+1, \beta(E,k,g(C)+3)+n_1\}\,.$$
	Then $\mc Q^{\rm tf}_{{\rm g},L}$
	is dense in $\mc Q^{\rm tf}_L$.
\end{lemma}
\begin{proof}
	As all components of $\mc Q_L$ have the same dimension,
	the same holds for the open subset $\mc Q^{\rm tf}_L$.
	Note that 
	$$\mc Q^{\rm tf}_L \setminus \mc Q^{\rm tf}_{{\rm g},L}=\mc Q^{\rm tf}_{L}\cap \mc Q_{\rm b}\,.$$
	The Lemma follows using \eqref{codimenion Q_{L,b}}.
\end{proof}

Combining the above results we have the following.

\begin{theorem}\label{Q_L is locally factorial}
	Let $k\geqslant 2,g(C)\geqslant 2$.
	Let $n_0$ be the smallest integer such that $kn_0>g(C)+1$.
	Let $n_1$ be the smallest integer such that $kn_1>g(C)+3$.
	Let 
	$$d\geqslant {\rm max}\{\beta(E,k,g(C)+1)+n_0+1, \beta(E,k,g(C)+3)+n_1\}\,.$$
	Then $\mc Q_L$ is a local complete intersection scheme which
	is also integral, normal and locally factorial.
\end{theorem}
\begin{proof}
	The Theorem follows using Proposition \ref{Q_L lci} once we 
	show that $\mc Q_L$ is irreducible. In view of  
	Lemma \ref{density Q^tf_L} and Lemma \ref{density  Q^tf_g,L}, 
	it suffices to show that $\mc Q^{\rm tf}_{{\rm g},L}$ is irreducible.
	
	Recall the notation from \S\ref{section stable quotients},
	in particular, the map $\Psi$ from \eqref{def Psi}. 
	This sits in the following commutative diagram whose maps 
	we describe next.
	\begin{equation}\label{cd Q^tf_g,L-1}
		\begin{tikzcd}[column sep=.7em]
			\mb U \arrow[rr,"\Psi"]\arrow[d]
			&& \mc Q^{\rm tf}_{\rm g}\arrow[d] \\
			R'\arrow[rr]&& {\rm Pic}^{d+kn}(C)
		\end{tikzcd}
	\end{equation}
	The bottom horizontal map sends a closed point $[x:\mc O_C^{\oplus N}\longrightarrow F]\in R'$ to
	${\rm det}(F)$. The right vertical map sends a closed 
	point $[q:E\longrightarrow F]\in \mc Q^{\rm tf}_{\rm g}$ to ${\rm det}(F)\otimes \mc O_C(knP)$.
	Let $L':=L\otimes \mc O_C(knP)$. 
	
	The bottom horizontal map in \eqref{cd Q^tf_g,L-1} is a smooth morphism. 
	This follows using Lemma \ref{description differential} and the reason explained 
	after \eqref{eq-smoothness-det}
	applied to the space $R'$. In particular, the morphism 
	$R'\longrightarrow {\rm Pic}^{d+kn}(C)$ is flat. Thus, $R'_{L'}$ is a smooth
	equidimensional scheme. Using \cite[Corollary 1.3]{Bhosle}  
	we easily see that $R'_{L'}$ is irreducible. 
	Taking the ``fiber'' of \eqref{cd Q^tf_g,L-1} over 
	the point $[L']\in {\rm Pic}^{d+kn}(C)$ we get 
	the following commutative diagram
	\[
	\begin{tikzcd}[column sep=.7em]
		\mb U_{L'} \arrow[rr,"\Psi_{L'}"]\arrow[d,"\Theta_{L'}",swap]
		&& \mc Q^{\rm tf}_{{\rm g},L}\arrow[d] \\
		R'_{L'}\arrow[rr]&& {[L']}
	\end{tikzcd}
	\]
	It follows that $\mb U_{L'}$ is irreducible. 
	By surjectivity of $\Psi$ on closed points 
	we get that $\Psi_{L'}$ is also surjective on closed points. 
	It follows that $\mc Q^{\rm tf}_{{\rm g},L}$
	is irreducible. This completes the proof of the Theorem. 
\end{proof}

Let $M^s_{k,L}$ denote the moduli space of stable 
bundles of rank $k$ and determinant $L$. 

\begin{theorem}\label{Picard group of Q_L}
	Let $r-k\geqslant2$. Assume one of the following two holds
	\begin{itemize}
		\item $k\geqslant 2$ and $g(C)\geqslant 3$, or
		\item $k\geqslant 3$ and $g(C)=2$.
	\end{itemize}
	Let $n_0$ be the smallest integer such that $kn_0>g(C)+1$.
	Let $n_1$ be the smallest integer such that $kn_1>g(C)+3$.
	Let 
	$$d\geqslant {\rm max}\{k\mu_0(E,k)+k,\beta(E,k,g(C)+1)+n_0+1, \beta(E,k,g(C)+3)+n_1\}\,.$$
	We have isomorphisms 
	$${\rm Pic}(\mc Q_L)\cong {\rm Pic}(M^s_{k,L})\times \Z \cong \Z \times \Z\,.$$
\end{theorem}

\begin{proof}
	The proof is similar to Theorem \ref{picard group of Quot} and so we only sketch it. 
	From (\ref{codimenion Q_{L,b}}) and the fact that 
	$\mc Q^{\rm tf}_L \setminus \mc Q^{\rm tf}_{{\rm g},L}=\mc Q^{\rm tf}_{L}\cap \mc Q_{\rm b}$ 
	it follows that 
	$${\rm dim}(\mc Q_L) - {\rm dim}(\mc Q_L\setminus \mc Q^{\rm tf}_{{\rm g},L})\geqslant 2\,.$$
	Now consider the diagram
	\[
	\begin{tikzcd}[column sep=.7em]
		\mb U_{L'} \arrow[rr,"\Psi_{L'}"]\arrow[d,"\Theta_{L'}",swap]
		&& \mc Q^{\rm tf}_{{\rm g},L}\arrow[d] \\
		R'_{L'}\arrow[rr]&& {[L']}
	\end{tikzcd}
	\]
	Just as in Lemma \ref{pic of Q and Q^s}, using \cite[Corollary 1.3]{Bhosle},
	and Lemma \ref{fiber-dimension-1} we have 
	$${\rm dim}(\mc Q^{\rm tf}_{{\rm g},L}) - 
		{\rm dim}(\mc Q^{\rm tf}_{{\rm g},L}\setminus \mc Q^s_L)\geqslant 2\,.$$
	Therefore we get that 
	$${\rm dim}(\mc Q_L) - {\rm dim}(\mc Q_L\setminus \mc Q^s_L)\geqslant 2\,.$$
	Since $\mc Q_L$ is locally factorial we have
		$${\rm Pic}(\mc Q_L)\cong {\rm Pic}(\mc Q^s_L)\,.$$
		Now we have the cartesian diagram
		\begin{equation}
	\xymatrix{
		\mb U^s_L\ar[r]^\Psi\ar[d]_{\Theta_L} & \mc Q^s_L\ar[d]^{\theta_L}\\
		R^s_L\ar[r]^{\psi}& M^s_{k,L}\,.
	}
\end{equation}
   	which we get by taking the fiber over $[L]$ of the diagram 
   	(\ref{diagram of principal bundles}). The rest of 
   	the proof is the same as the proof of 
   	Theorem \ref{picard group of Quot}, by considering 
   	this diagram instead of (\ref{diagram of principal bundles}). 
   	The second equality follows from \cite[Theorem B]{DN}.
\end{proof}

\section{Quot Schemes ${\rm Quot}_{C/\C}(E,1,d)$}

In this section we consider the case $k=1$. We only sketch the proofs
as they are similar to the earlier cases considered. 

\begin{theorem}\label{k=1}
	Let $k=1$. Let $d\geqslant {\rm max}\{\mu_0(E,1)+1,\beta(E,1,g(C)+3)+g(C)+4\}$. Then 
	$${\rm Pic}(\mc Q)\cong {\rm Pic}({\rm Pic}^d(C))\times \Z\times \Z\,,\quad \quad 
		{\rm Pic}(\mc Q_L)\cong \Z\times \Z\,.$$
\end{theorem}
\begin{proof}
	We can apply Theorem \ref{det flat} to conclude that $\mc Q$ is 
	integral, normal and locally factorial.
	We claim that $\mc Q^{\rm tf}$ is smooth. To see this, let 
	$$0\longrightarrow S\longrightarrow E\longrightarrow L\longrightarrow 0$$
	be a quotient. Applying ${\rm Hom}(-,L)$ we get a surjection 
	${\rm Ext}^1(E,L)\longrightarrow {\rm Ext}^1(S,L)\longrightarrow 0$.
	By Lemma \ref{lemma to define good locus} it follows that ${\rm Ext}^1(E,L)=0$.
	It easily follows that $\mc Q^{\rm tf}$ is smooth.
	
	Let 
	\begin{equation}
		\rho_1:C\times {\rm Pic}^{d}(C)\longrightarrow C\,,\quad \quad 
		\rho_2:C\times {\rm Pic}^{d}(C)\longrightarrow {\rm Pic}^{d}(C)
	\end{equation}
	be the projections.
	Let $\mc L$ be a Poincare bundle on $C\times {\rm Pic}^{d}(C)$. Define 
	$$\mc E:=\rho_{2*}[\rho_{1}^*E^\vee\otimes \mc L]\,.$$ 
	Using Lemma \ref{lemma to define good locus} and cohomology and base change 
	we easily conclude that 
	$\mc E$ is a locally free sheaf on ${\rm Pic}^{d}(C)$ such that 
	the fibre over the point $[L]\in {\rm Pic}^{d}(C)$ is isomorphic 
	to ${\rm Hom}(E,L)$. Let $\mb W\subset \mb P(\mc E^\vee)$ be the open subset 
	consisting of points parametrizing surjective maps. Both $\mb W$ and 
	$\mc Q^{\rm tf}$ are smooth. There is a map
	$\mb W\longrightarrow {\mc Q}^{\rm tf}$ which is bijective on points
	(and hence an isomorphism as both are smooth)
	and sits in a commutative diagram
	\[
	\xymatrix{
		\mb W\ar[r]^\sim\ar[dr]& \mc Q^{\rm tf}\ar[d]^{\rm det}\\
		& {\rm Pic}^{d}(C)
	}
	\]
	Using Remark \ref{remark k=1} it follows that 
	${\rm dim}(\mb P(\mc E^\vee))-{\rm dim}(\mb P(\mc E^\vee)\setminus \mb W)\geqslant 2$.
	Thus, it follows that ${\rm Pic}(\mc Q^{\rm tf})\cong {\rm Pic}(\mb W)\cong 
	{\rm Pic}(\mb P(\mc E^\vee))\cong {\rm Pic}({\rm Pic}^d(C))\times \Z$.
	By Lemma \ref{lemma 2 codim}, $\mc Q\setminus\mc Q^{\rm tf}=\bar Z_1$  
	is irreducible of codimension 1 and so 
	we have an exact sequence 
	$$0\longrightarrow \mb Z\longrightarrow {\rm Pic}(\mc Q)\longrightarrow
		 {\rm Pic}(\mc Q^{\rm tf})\longrightarrow 0\,.$$
	It easily follows that we have an isomorphism 
	$${\rm Pic}(\mc Q)\cong {\rm Pic}({\rm Pic}^d(C))\times \Z\times \Z\,.$$

	For $\mc Q_L$, we first 
	show that $\mc Q_L$ is integral, normal and locally factorial.
	This is easily done using Proposition \ref{Q_L lci},
	Lemma \ref{density Q^tf_L} and using the fact
	that $\mc Q^{\rm tf}_L\cong \mb W_L$. The rest of 
	the proof follows in the same way as that of $\mc Q$,
	once we use the irreducibility of $\bar Z_{1,L}$
	and the fact that it is of codimension 1, see 
	\eqref{codim of Q^{tf}}. We remark that when $k=1$,
	unlike in Theorem \ref{Q_L is locally factorial}, we do not 
	need to use \cite{Bhosle} and hence do not need the 
	hypothesis that $g(C)\geqslant 2$. 
\end{proof}


\end{document}